\documentclass[12pt]{amsproc}
\usepackage[ansinew]{inputenc}
\usepackage{graphicx}
\usepackage{color}
\usepackage[colorlinks]{hyperref}

\usepackage{enumerate,latexsym}
\usepackage{latexsym}
\usepackage{amsmath,amssymb}
\usepackage{graphicx}
\usepackage{soul}
\usepackage{times}
\newfont{\bb}{msbm10 at 11pt}
\newfont{\bbsmall}{msbm8 at 8pt}

\def\rth{\mathbb{R}^3}
\def\R{\mathbb{R}}

\def\N{\mathbb{N}}
\def\Z{\mathbb{Z}}
\def\C{\mathbb{C}}

\def\Hip{\mathbb{H}}

\def\D{\mathbb{D}}

\def\esf{\mathbb{S}}

\newcommand{\la}{\looparrowright}

\newcommand{\ben}{\begin{enumerate}}
\newcommand{\bit}{\begin{itemize}}
\newcommand{\een}{\end{enumerate}}
\newcommand{\eit}{\end{itemize}}
\newcommand{\wh}{\widehat}
\newcommand{\Int}{\mbox{\rm Int}}

\newcommand{\wt}{\widetilde}

\newcommand{\sol}{\mbox{Sol}_3}
\newcommand{\ed}{\end{document}}
\newcommand{\ov}{\overline}

\definecolor{rr}{rgb}{.8,0,.3}
\definecolor{drr}{rgb}{.5,0,.4}
\definecolor{rrr}{rgb}{.9,0,.1}

\def\a{{\alpha}}

\def\t{{\theta}}

\def\g{{\gamma}}
\def\G{{\Gamma}}
\def\l{{\lambda}}

\def\de{{\delta}}
\def\be{{\beta}}
\def\ve{{\varepsilon}}

\def\cM{\mathcal{M}}

\def\cF{\mathcal{F}}

\def\cW{\mathcal{W}}

\let\8=\infty \let\0=\emptyset

\let\alfa=\alpha

\let\parc=\partial

\def\esiz{\langle}
\def\esde{\rangle}

\def\cte.{\mathop{\rm cte.}\nolimits}

\def\det{\mathop{\rm det}\nolimits}

\def\div{\mathop{\rm div }\nolimits}

\def\cosh{\mathop{\rm cosh }\nolimits}

\def\N{\mathbb{N}}

\def\R{\mathbb{R}}
\def\Z{\mathbb{Z}}
\def\C{\mathbb{C}}

\def\D{\mathbb{D}}

\def\H{\mathbb{H}}
\def\S{\mathbb{S}}

\newtheorem{theorem}{Theorem}[section]
\newtheorem{lemma}[theorem]{Lemma}
\newtheorem{proposition}[theorem]{Proposition}

\newtheorem{remark}[theorem]{Remark}
\newtheorem{corollary}[theorem]{Corollary}
\newtheorem{definition}[theorem]{Definition}
\newtheorem{conjecture}[theorem]{Conjecture}
\newtheorem{assertion}[theorem]{Assertion}

\newcommand{\su}{\mbox{SU}(2)}
\newcommand{\EE}{\wt{\mbox{E}}(2)}

\renewcommand{\sl}{\wt{\mbox{SL}}(2,\R)}

\numberwithin{equation}{section}

\usepackage{fullpage}

\begin{document}

\begin{title}
{The geometry of stable minimal surfaces in metric Lie groups}
\end{title}
\today
\author{William H. Meeks III}
\address{William H. Meeks III, Mathematics Department,
University of Massachusetts, Amherst, MA 01003}
\email{profmeeks@gmail.com}
\author{Pablo Mira}
\address{Pablo Mira, Department of
Applied Mathematics and Statistics, Universidad
Polit\'ecnica de Cartagena, E-30203 Cartagena, Murcia, Spain.}

\email{pablo.mira@upct.es}

\author{Joaqu\'\i n P\'erez}
\address{Joaqu\'\i n P\'erez, Department of Geometry and Topology {and Institute of Mathematics IEMath-GR},
University of Granada, 18001 Granada, Spain}
 \email{jperez@ugr.es}

\begin{abstract}
We study geometric properties of compact stable minimal surfaces with boundary
in homogeneous 3-manifolds $X$ that can be expressed as a semidirect product
of $\R^2$ with $\R$ endowed with a left invariant metric. For any such compact
minimal surface $M$, we provide a priori radius estimate which depends only on
the maximum distance of points of the boundary $\partial M$ to a vertical geodesic
of $X$. 
We also give a
generalization of the classical Rado's Theorem~\cite{ra2} in $\R^3$ to the context
of compact minimal surfaces with graphical boundary over a convex horizontal
domain in $X$, and we study the geometry, existence and uniqueness of this type of Plateau
problem.

\par
\vspace{.1cm} \noindent{\it Mathematics Subject Classification:}
Primary 53A10, Secondary 49Q05, 53C42.

\vspace{.1cm} \noindent {\em Key words and phrases:} { Minimal surface,
radius estimates, stability, Rado's Theorem,
metric Lie group, homogeneous 3-manifold, left invariant metric.
}

\end{abstract}
\maketitle

\vspace{-.4cm}

\section{Introduction.}  \label{sec:introduction}
In this paper we study the geometry of compact minimal surfaces with boundary in homogeneous manifolds
diffeomorphic to $\R^3$. By classification, each such homogeneous manifold $X$ is a \emph{metric Lie group}, i.e.,
a simply connected 3--dimensional Lie group equipped with a left invariant Riemannian metric $\esiz,\esde$. For such
an $X$ there are two possibilities; either $X$ is isometric to the universal cover of the special linear
group ${\rm SL}(2,\R)$ endowed with some left invariant metric, or $X$ is a \emph{metric semidirect product}.
By definition, a metric semidirect product $X=\R^2\rtimes_A \R$ is given as a Lie group $(\R^3\equiv \R^2 \times \R,*)$
together with a certain left invariant metric (its \emph{canonical metric}, see Definition~\ref{def2.1}),
where the product operation $*$ is
expressed in terms of some real $2\times2$ matrix $A\in \cM_2(\R)$ as
\begin{equation*}
 ({\bf p}_1,z_1)*({\bf p}_2,z_2)=({\bf p}_1+ e^{z_1 A}\  {\bf
 p}_2,z_1+z_2);
\end{equation*}
see Subsection~\ref{subsem} for more details. When ${\rm trace}(A)=0$, $X$ is a \emph{unimodular}
semidirect product; typical examples of Riemannian manifolds in this situation are the Euclidean space $\R^3$, the
Heisenberg space ${\rm Nil}_3$ or the solvable Lie group ${\rm Sol}_3$ with its usual Thurston geometry.
When $\mathrm{trace}(A)\neq 0$, we obtain the \emph{non-unimodular} semidirect products, among which we highlight
the hyperbolic space $\H^3$ and the Riemannian product $\H^2\times \R$.

The geometry of minimal surfaces in homogeneous 3-manifolds of
non-constant
sectional curvature has been deeply studied in the last decade, specially in
the case that the isometry group of the homogeneous manifold has dimension four.
To indicate just a few relevant works in this area, we may cite
\cite{AbRo1,AbRo2,cor2,da1,da7,dh1,fm3,fm1,
mper1,ner2,MMRo4,rose2,tor2}.
An outline of the beginning of the theory of constant mean curvature surfaces
in homogeneous 3-manifolds with a 4-dimensional
isometry  group can be consulted in~\cite{dhm1,fm4}.

For the generic, \emph{non-symmetric} case of homogeneous
3-manifolds with
an isometry group of dimension three, the theory of minimal surfaces is less
developed. For some works dealing with this more general situation, see
e.g., \cite{dmr1,dm2,des,iw1,Kru1,lomu,mmpr4,mpe17,mpe11,men,Ng,ra1}. For an
introduction to the geometry of general simply connected homogeneous
3-manifolds, see~\cite{mpe11}.

In this paper we develop some aspects of the theory of compact  minimal
surfaces with boundary in metric semidirect products $X=\R^2\rtimes_A \R$.
We shall be specially interested in the geometry, existence and uniqueness of
solutions to the Plateau problem for graphical boundaries on convex
{domains of $\R^2\rtimes_A\{ 0\} $,} and on estimating the radius of compact stable minimal surfaces
with boundary in metric semidirect products. We recall that the \emph{radius}
of a compact Riemannian surface $M$ with boundary is the maximum distance of
points in the surface to its boundary $\parc M$.

An important classical result of Rado~\cite{ra2} states that a simple closed
curve $\G$ in $\rth$ that  has a 1-1 orthogonal projection to a convex curve
in a plane $P\subset \rth$, is the boundary of a minimal disk of finite area
that is a graph over its projection to $P$, and furthermore, any branched
minimal disk in $\rth$ with boundary $\G$ has similar properties.
Other classical Rado type results for minimal surfaces in $\R^3$ were obtained
in~\cite{me1,ni3}. In Section~\ref{section:Plateau} of this paper we will
extend Rado's theorem to the context of minimal surfaces in metric semidirect
products; see Theorems~\ref{Plateau:unique} and ~\ref{lemma2}.

In Section~\ref{sec:4}, Theorems~\ref{Plateau:unique} and ~\ref{lemma2} are used
to study the geometry of compact minimal surfaces $\Sigma$ in a non-unimodular semidirect
product, such that $\Sigma $ is the boundary of a round Euclidean circle in $\R^2\rtimes_A \{0\}$;
namely we prove that as the radius of such a circle goes to infinity, then the angles that
$\Sigma$ makes with $ \R^2\rtimes_A \{0\}$ along its boundary circle
converge uniformly to $\pi/2$; see Theorem~\ref{thm:Plateau}
for a generalization of this result and also see the related application given in Corollary~\ref{corol5.4}
to the existence of a minimal annulus bounded by two circles in $ \R^2\rtimes_A \{0\}$
of large radius, so that these circles can be taken arbitrarily far away from each other.
All of these results  are then applied in Section~\ref{sec:estim} to obtain
radius estimates of compact minimal surfaces with boundary in metric
semidirect products $X=\R^2\rtimes_A \R$, as we explain next. Given $A\in \mathcal{M}_2(\R)$,
any vertical line $\Gamma=\{(x_0,y_0,z) \mid z\in \R\}$
in $X=\R^2\rtimes_A \R$ is a geodesic of $X$ (endowed with its canonical metric),
which we call a \emph{vertical geodesic}. By a \emph{metric solid cylinder}
of radius $r>0$ in $X$ around $\Gamma$ we mean the set of points $\cW(\G ,r)$ in $X$ whose distance to $\Gamma$
is at most $r$. With these definitions in mind, the next theorem summarizes another main result of the paper.

\begin{theorem}
\label{th:intro}
Let $X=\R^2\rtimes_A \R$ be a metric semidirect product, and let $\cW(\G ,r)$
be a solid metric cylinder in $X$ of radius $r>0$ around a vertical geodesic $\Gamma$.
There exists some $R=R(r)>0$ such that if $M$ is a compact, stable minimal surface
in $X$ whose boundary $\partial M$ is contained in $\cW(\G ,r)$, then $M$ has radius at most $R$.
In particular, there are no complete stable minimal surfaces contained in $\cW(\G ,r)$.
\end{theorem}


Theorem~\ref{th:intro} will be proved in Section~\ref{sec:estim};
see Theorem~\ref{corrad2}.
Another tool used to prove Theorem~\ref{corrad2}
is Proposition~\ref{Plateau2}, where we will construct a certain family of mean
convex solid cylinders over appropriately defined ellipses in non-unimodular
metric semidirect products with positive Milnor $D$-invariant. For this and
other purposes, we will prove in the
Appendix a few additional technical results about the geometry of these metric semidirect products.

\section{Background material on 3-dimensional metric Lie groups.}
 \label{sec:background}

This preliminary section is devoted to state some basic properties of
 3-dimensional Lie groups endowed
with a left invariant metric that will be used
freely in later sections. For details
of these basic properties, see the general reference~\cite{mpe11}.

Let $Y$ denote a simply connected, homogeneous Riemannian
3-manifold, and  assume that it is not isometric to the
Riemannian product of the 2-sphere
$\S^2(\kappa)$ of constant curvature $\kappa >0$ with the real line.
Then $Y$ is isometric to a simply connected, 3-dimensional Lie group
$G$ equipped with a left invariant metric $\langle ,\rangle $,
i.e., for every $p\in G$, the left translation $l_p\colon G\to G$, $l_p(q)=p\, q$, is an isometry
of $\langle ,\rangle $.
We will call such a space a
\emph{metric Lie group}, $X=(G,\langle ,\rangle )$. When $X$ is simply connected, there are three possibilities:

\begin{enumerate}[$\bullet $]
\item $X$ is isometric to the special unitary group $\su$ with a left invariant metric.
This is the only case
in which $X$ is not diffeomorphic to $\R^3$, and the family of
left invariant metrics is 3-dimensional.
\item
$X$ is isometric to the universal cover $\sl$ of the special linear group, equipped
with a left invariant metric. Again,
there is a 3-dimensional family of such metrics.
\item
$X$ is isometric to a semidirect product $\R^2\rtimes_A \R$ equipped
with its {\it canonical metric,} which is the left
invariant metric introduced in Definition~\ref{def2.1} below.
In this third case, the underlying Lie group is
$(\R^3\equiv \R^2\times \R,*)$, where the group operation $*$ is
expressed in terms of some real $2\times2$ matrix $A\in \cM_2(\R)$ as
\begin{equation}
\label{eq:5}
 ({\bf p}_1,z_1)*({\bf p}_2,z_2)=({\bf p}_1+ e^{z_1 A}\  {\bf
 p}_2,z_1+z_2);
\end{equation}
\end{enumerate}
here ${\bf p}_1,{\bf p}_2\in \R^2$, $z_1,z_2\in \R$ and
$e^B=\sum _{k=0}^{\infty }\frac{1}{k!}B^k$ denotes the usual exponentiation
of a matrix $B\in \cM_2(\R )$.

\subsection{Semidirect products.}
\label{subsem}
Consider the semidirect product $\R^2\rtimes_A
\R$, where
\begin{equation} \label{equationgenA} A=\left(
\begin{array}{cr}
a & b \\
c & d \end{array}\right) .
\end{equation}
Then, in terms of the coordinates $(x,y)\in \R^2$, $z\in \R $, we
have the following basis $\{ F_1,F_2,F_3\} $ of the
{linear} space of {\it right invariant}
vector fields on $\R^2\rtimes_A
\R$:
\begin{equation}
\label{eq:6}
 F_1=\partial _x,\quad F_2=\partial _y,\quad F_3(x,y,z)=
(ax+by)\partial _x+(cx+dy)\partial _y+\partial _z.
\end{equation}
In the same way, a {\it left invariant} frame $\{ E_1,E_2,E_3\} $ of $X$
is given by
\begin{equation}
\label{eq:6*}
 E_1(x,y,z)=a_{11}(z)\partial _x+a_{21}(z)\partial _y,\quad
E_2(x,y,z)=a_{12}(z)\partial _x+a_{22}(z)\partial _y,\quad
 E_3=\partial _z,
\end{equation}
where
\begin{equation}
\label{eq:exp(zA)}
 e^{zA}=\left(
\begin{array}{cr}
a_{11}(z) & a_{12}(z) \\
a_{21}(z) & a_{22}(z)
\end{array}\right) .
\end{equation}
In terms of $A$, the Lie bracket relations are:
\begin{equation}
\label{eq:8a} [E_1,E_2]=0,
\quad
[E_3,E_1]=aE_1+cE_2,
\quad
 [E_3,E_2]=bE_1+dE_2.
 \end{equation}
Observe that ${\rm Span}\{ E_1,E_2\}$ is an integrable
2-dimensional distribution of $\R^2\rtimes_A
\R$, whose integral surfaces are the leaves of the
foliation $\mathcal{F}= \{ \R^2\rtimes _A\{ z\} \mid z\in \R \} $ of
$\R^2\rtimes _A\R$.

\begin{definition}
 \label{def2.1}
 {\rm
We define the {\it canonical left invariant metric} on the
semidirect product $\R^2\rtimes _A\R $ to be that one for which the
left invariant basis $\{ E_1,E_2,E_3\} $ given by (\ref{eq:6*}) is
orthonormal. Equivalently, it is the left invariant extension to $\R^2\rtimes _A\R $
of the inner product on the tangent space $T_{\vec{0}} (\R^2\rtimes_A
\R)$ at the identity element $\vec{0}=(0,0,0)$
that makes $\{(\partial _x)_{\vec{0}},(\partial _y)_{\vec{0}},(\partial _z)_{\vec{0}}\}$
an orthonormal basis.}
\end{definition}

We next emphasize some other metric properties of the canonical left
invariant metric $\langle ,\rangle $ on $\R^2\rtimes_A \R$:

\begin{enumerate}[$\bullet $]
\item The mean curvature of each leaf of the foliation $\mathcal{F}=
\{ \R^2\rtimes _A\{ z\} \mid z\in \R \}$ with respect to the unit
normal vector field $E_3$ is the constant $H=\mbox{trace}(A)/2$. All
the leaves of the foliation $\mathcal{F}$ are intrinsically flat.
\item The change from the orthonormal basis $\{ E_1,E_2,E_3\} $ to the
basis $\{ \partial _x,\partial _y,\partial _z\} $ given by
(\ref{eq:6*}) produces the following expression for the metric
$\langle ,\rangle $ in the $x,y,z$ coordinates of $X:=(\R^2\rtimes_A
\R,\esiz,\esde)$:
\begin{equation}
\label{eq:13}
 \left.
\begin{array}{rcl}
\langle ,\rangle & =&
 \left[ a_{11}(-z)^2+a_{21}(-z)^2\right] dx^2+
\left[ a_{12}(-z)^2+a_{22}(-z)^2\right] dy^2 +dz^2 \\
& + & \rule{0cm}{.5cm} \left[
a_{11}(-z)a_{12}(-z)+a_{21}(-z)a_{22}(-z)\right] \left( dx\otimes
dy+dy\otimes dx\right)
\\
&=& \rule{0cm}{.5cm} e^{-2\mbox{\footnotesize trace}(A)z} \left\{
 \left[ a_{21}(z)^2+a_{22}(z)^2\right] dx^2+
\left[ a_{11}(z)^2+a_{12}(z)^2\right] dy^2\right\} +dz^2 \\
& - & \rule{0cm}{.5cm}
 e^{-2\mbox{\footnotesize trace}(A)z} \left[
a_{11}(z)a_{21}(z)+a_{12}(z)a_{22}(z)\right] \left( dx\otimes
dy+dy\otimes dx\right) .
\end{array}
\right.
\end{equation}

\item The Levi-Civita connection associated to the canonical left invariant
metric is easily deduced from the Koszul formula and~(\ref{eq:8a}) as follows:
{
\begin{equation}
\label{eq:12}
\begin{array}{l|l|l}
\rule{0cm}{.5cm}
\nabla _{E_1}E_1=a\, E_3 & \nabla _{E_1}E_2=\frac{b+c}{2}\, E_3
& \nabla _{E_1}E_3=-a\, E_1-\frac{b+c}{2}\, E_2 \\
\rule{0cm}{.5cm}
\nabla _{E_2}E_1=\frac{b+c}{2}\, E_3 & \nabla _{E_2}E_2=d\, E_3
& \nabla _{E_2}E_3=-\frac{b+c}{2}\, E_1-d\, E_2 \\
\rule{0cm}{.5cm}
\nabla _{E_3}E_1=\frac{c-b}{2}\, E_2 & \nabla
_{E_3}E_2=\frac{b-c}{2}\, E_1 & \nabla _{E_3}E_3=0.
\end{array}
\end{equation}
}
\end{enumerate}

\begin{remark}
\label{rem3.2}
{\rm
It follows from equation (\ref{eq:13}) that given $(x_0,y_0)\in \R^2$, the map
$$(x,y,z)\stackrel{\phi }{\mapsto }(-x+2x_0,-y+2y_0,z)$$ is an
isometry of $(\R^2\rtimes _A\R ,\langle ,\rangle )$ into itself.
Note that $\phi $ is the rotation by angle $\pi $ around the line
$l=\{ (x_0,y_0,z)\ | \ z\in \R \} $, and the fixed point set of
$\phi $ is the geodesic $l$. In particular, vertical lines in the
$x,y,z$-coordinates of $\R^2\rtimes_A\R$ are geodesics of its canonical
metric, which are the axes or fixed point sets of the isometries
corresponding to rotations by angle $\pi$ around them.
For any line $L$ in $\R^2\rtimes_A \{0\}$, let $P_L$ denote the vertical plane
$\{(x,y,z) \mid (x,y,0)\in L, \, z\in \R\}$ containing the set of
vertical lines passing though $L$. It follows that the plane $P_L$
is  ruled by vertical geodesics and furthermore,  since the rotation by
angle $\pi$ around any vertical line in $P_L$ is an isometry that leaves
$P_L$ invariant, then $P_L$ has zero mean curvature. Thus, every
metric Lie group that can be expressed as a semidirect product of
the form $\R^2\rtimes_A\R$ with its canonical metric has many
minimal foliations by parallel vertical planes, where by parallel we
mean that the related lines in $\R^2\rtimes_A\{0\}$ for these planes
are parallel in the intrinsic metric.
}
 \end{remark}

\subsection{Unimodular groups.}
\label{secunimodgr}
Among all simply connected, 3-dimensional Lie groups, the cases
$\su$, $\sl$, 
$\sol$ (whose underlying group arises in the so called
{\it Sol geometry}), $\EE $
(universal cover of the Euclidean group of orientation-preserving rigid motions of the plane),
Nil$_3$ (Heisenberg group) and $\R^3$
comprise the \emph{unimodular} Lie groups; the cases of $\sol$,
$\widetilde{\mbox{E}}(2)$, Nil$_3$ and $\R^3$ with their left invariant metrics correspond to the
metric semidirect products $\R^2 \rtimes_A\R$, where the trace of
$A$ is zero.
We refer the
reader to~\cite{mpe11} for further details.

\subsection{Non-unimodular groups.}
The case $X=\R^2\rtimes_A \R$ with ${\rm trace} (A)\neq 0$
corresponds to the simply connected, 3-dimensional, {\em non-unimodular}
metric Lie groups. In this case,
up to the rescaling of the metric of $X$, we may assume that
${\rm trace} (A)=2$. {\it This normalization in the non-unimodular case
will be assumed from now on throughout the paper}.
After an appropriate orthogonal change of the left invariant frame that fixes
the vertical field $E_3$,
we may express the matrix $A$ uniquely as (see Section 2.5 in~\cite{mpe11}):
\begin{equation}
\label{Axieta} A=A(\alfa,\beta)= \left(
\begin{array}{cc}
1+\alfa & -(1-\alfa)\beta\\
(1+\alfa)\beta & 1-\alfa \end{array}\right), \hspace{1cm} \alfa,\beta\in [0,\infty).
\end{equation}

The \emph{canonical basis} of the
non-unimodular metric Lie group $X$ is, by definition, the
left invariant orthonormal frame $\{E_1,E_2,E_3\}$ given in
\eqref{eq:6*} by the matrix $A$ in~(\ref{Axieta}).
In other words, every simply connected, non-unimodular metric Lie
group is isomorphic and isometric (up to possibly rescaling the
metric) to $\R^2\rtimes _A\R $ with its canonical metric, where $A$
is given by (\ref{Axieta}). If $A=I_2$ where $I_2$ is the identity
matrix, we get a metric Lie group that we denote by $\H^3$, which is
isometric to the  hyperbolic 3-space with its standard metric of
{constant sectional curvature $-1$ and
where the underlying Lie group structure is} isomorphic to that of
the set of similarities of $\R^2$.
Under the assumption that $A\neq I_2$, the
determinant of $A$ determines uniquely the Lie group structure.

\begin{definition} \label{milnor-invariant}
{\rm
The \emph{Milnor $D$-invariant} of $X=\R^2\rtimes_A \R$
is the determinant of $A$:
\begin{equation}
\label{Dxieta} D=(1-\alfa^2)(1+\beta^2) ={\rm det} (A).
\end{equation}
}
\end{definition}
\vspace{.2cm}

Assuming $A\neq I_2$, given $D\in \R $, one can solve
(\ref{Dxieta}) for $\alfa=\alfa(D,\beta)$, producing a related matrix $A(D,\beta)$
by equation (\ref{Axieta}), and the space of canonical left
invariant metrics on the corresponding non-unimodular Lie group
structure is parameterized by the values of $\beta \in [m(D),\infty )$,
where
\begin{equation}
\label{eq:m(D)}
m(D)=\left\{ \begin{array}{cl}
\sqrt{D-1} & \mbox{ if $D>1$,}\\
0 & \mbox{ otherwise}.
\end{array}\right.
\end{equation}
In particular, after scaling so that ${\rm trace} (A)=2$ and assuming that $A\neq I_2$, the
space of simply connected, 3-dimensional, non-unimodular metric Lie groups with a given
$D$-invariant is 1-dimensional.

\begin{remark}\label{clasemi}
\emph{From now on, by a metric semidirect product $X$ we will mean
(without loss of generality, see the explanation below) a semidirect product
$\R^2\rtimes_A \R$ endowed with its canonical left invariant metric
$\esiz,\esde$, and such that the matrix $A\in \mathcal{M}_2(\R)$ either has
trace zero (unimodular case) or is given by expression \eqref{Axieta} for some
$\alfa,\beta \in [0,\8)$ (non-unimodular case). We must observe that we do not
lose any generality with this normalization, since by the previous discussion,
every metric semidirect product $\R^2\rtimes_B \R$ whose associated matrix $B$
has non-zero trace is both isomorphic and isometric (after an adequate
rescaling) to a metric semidirect product $\R^2\rtimes_A \R$ where $A$ is
given by \eqref{Axieta}. Moreover, the corresponding isomorphism takes the
horizontal foliation $\{ \R^2\rtimes_B \{z\}\ | \ z\in \R \} $ of
$\R^2\rtimes_B \R$ to the horizontal foliation $\{ \R^2\rtimes_A \{z\} \ | \
z\in \R \} $ of $\R^2\rtimes_A \R$ and also
preserves the left invariant vertical vector fields $E_3$ of their respective
canonical frames.
}
\end{remark}
Along the paper, we will denote by $\Pi\colon \R^2\rtimes_A\R\to \R^2\rtimes_A\{0\}$ the
projection $\Pi(x,y,z)=(x,y,0)$.

\section{Rado's Theorem in metric semidirect products.}
 \label{section:Plateau}
In this section we prove some results concerning the geometry of
solutions to Plateau type problems in metric semidirect products $X=\R^2
\rtimes_A \R$, when there is some geometric constraint on the
boundary values of the solution. The first of these
results is Theorem~\ref{Plateau:unique} below.
We remark that several versions of this theorem
in the classical setting of $X=\rth=\R^2\times \R$ were proved by
Rado~\cite{ra2}, Nitsche~\cite{ni3} and  Meeks~\cite{me1}.
We point out that one of the difficulties in obtaining Rado-type
results in the situation below is that the vertical translation
$(x,y,z)\mapsto (x,y,z+t)$ might not be an isometry of the canonical metric on
$\R^2\rtimes _A\R $.

\begin{theorem}[Rado's Theorem in metric semidirect products]
\label{Plateau:unique}
 Let $X=\R^2\rtimes_A \R$ be a metric semidirect product.
 Suppose that $E$ is a compact convex
disk in $\R^2\rtimes_A\{0\}$, $C=\partial E$ and $\G\subset
\Pi^{-1}(C)$ is a continuous simple closed curve such that
$\Pi|_{\G }\colon \Gamma \to C$ monotonically parameterizes\footnote{This
means that for every point $p\in C$, $\Pi^{-1}(p)\cap \G$ is a
compact interval or a single point.} $C$. Then:

\ben[1.]
\item $\G$ is the boundary of a compact embedded disk $D$ of finite least area.
\item The interior of $D$ is a smooth $\Pi$-graph over the interior of $E$.
\een
\end{theorem}

Theorem~\ref{Plateau:unique} will be a direct consequence of Theorem~\ref{lemma2} below,
which actually gives a more complete statement. The proof of Theorem~\ref{lemma2} also depends
on the following Lemma~\ref{lemma1}; both of these results will also be used
in the proof of Theorem~\ref{thm:Plateau}  in
Section~\ref{sec:4}.

\begin{lemma}
\label{lemma1}
 Suppose $X=\R^2\rtimes_A \R$ is a metric semidirect product. Let $E\subset \R^2\rtimes_A\{0\}$ be a compact
convex disk with boundary curve $C$. If $M$ is a compact branched
minimal surface in $X$ with boundary contained in $\Pi^{-1}(C)$,
then:
\ben[1.]
\item  $\Int(M)$ is contained in the interior of $\Pi^{-1}(E)$.
\item If $\partial M$ is of class $C^2$, then $M$ is an
immersion near its boundary and
transverse to $\Pi^{-1}(C)$ along $\partial M$. \een \end{lemma}
\begin{proof}
The proof of this lemma uses the fact stated in
Remark~\ref{rem3.2} that for every line $L\subset
\R^2 \rtimes_A \{0\}$, the vertical plane $\Pi^{-1}(L)$ has zero
mean curvature.

Suppose that $M$ is a compact branched minimal surface with
$\partial M \subset \Pi ^{-1}(C)$ and we will prove the first item in the
lemma. Arguing by contradiction, assume
there exists a point $p\in  {\rm Int}(M)$ which is not contained in
the interior of $\Pi^{-1}(E)$. Since $E$ is convex, there exists a
line  $L\subset \R^2 \rtimes_A \{0\}$ such that  $\Pi(p)\in L$ and
$L$ is disjoint from ${\rm Int}(E)$. Hence the vertical minimal
plane $\Pi^{-1}(L)$ intersects $\Int(M)$ at $p$ and so, by the
maximum principle, $M$ contains interior points on both sides of
$\Pi^{-1}(L)$ near $p$.

Consider the product foliation ${\mathcal F}(L)=\{L_{t}\}_{t\in \R}$
of lines in $\R^2\rtimes_A\{0\}$ parallel to $L=L_0$ and
parameterized so that $E \subset \cup_{t\leq 0}L_t$. Let
$\{\Pi^{-1}(L_t)\}_{t\in\R}$ be the related foliation of $X$ by
minimal vertical planes. By compactness of $M$, there is a largest
value $t_0>0$, such that $\Pi^{-1}(L_{t_0})\cap M \not = \mbox{\rm \O}$. But at
any point of this non-empty intersection, we obtain a contradiction
to the maximum principle applied to the minimal surfaces
$\Pi^{-1}(L_{t_0})$ and $M$. This contradiction proves item~1 of
the lemma. Item~2 of the lemma follows from Theorem~2 in~\cite{my2}.
\end{proof}

\begin{theorem}
\label{lemma2}
 Let $X=\R^2\rtimes_A\R$ be a metric semidirect product,  $E$ be a compact convex
 disk in $\R^2\rtimes_A\{0\}$
and $C=\partial E$. Suppose $\G\subset \Pi^{-1}(C)$ is a continuous
simple closed curve such that the projection $\Pi\colon \Gamma \to C$ monotonically
parameterizes $C$. Let $W=\Pi^{-1}(E)$. Then:
\ben[1.]
\item If $D$ is a compact, branched minimal disk in $X$ with $\partial D=\G$, then
the following properties hold:
\begin{description}
\item[{\it 1a}] $D$ is an embedded disk.
\item[{\it 1b}] The interior of $D$ is a smooth $\Pi$-graph over the interior of $E$,
i.e., $\Pi|_{\mbox{\small \rm Int}(D)}\colon {\rm Int}(D)\to
\Int(E)$ is a diffeomorphism.
\item[{\it 1c}] If \,${\Pi|_{\G }}\colon \Gamma \to C$  is a homeomorphism, then
${\Pi|_D}\colon D\to E$
is a homeomorphism.
\item[{\it 1d}] If \,$\G$ is of class $C^2$, then the inclusion map of $D$ is an immersion
along $\partial D$ and $D$ is transverse to $\Pi ^{-1}(C)$
along $\G $.
\item[{\it 1e}] If \,${\Pi|_{\G}}\colon \Gamma \to C$  is a diffeomorphism,
then ${\Pi|_D}\colon D\to E$
is a diffeomorphism.
\end{description}
\item There exist compact minimal disks $D_L, D_T,D_B$ in $W$
with boundary $\G$ such that
\ben[2a:]
\item $D_L$ is an embedded disk of finite least area in $X$.
\item $D_T$ is an embedded disk of finite least area in the closed region of $W$ above the graph $D_T$.
\item $D_B$ is an embedded disk of finite least area in the closed region of $W$ below the graph $D_B$.
\item Any compact branched minimal surface $M$ in $X$ whose boundary lies in the compact
set $W(D_T,D_B)\subset W$ between the graphs $D_T$ and $D_B$,
satisfies $M\subset W(D_T,D_B)$. In particular, the disks $D_T$ and $D_B$ are uniquely
defined by Properties 2b, 2c and 2d; hence,  $\G$ is the boundary of
a unique compact branched minimal surface if and only if $D_T=D_B$.
\een
\een
\end{theorem}
\begin{proof}
We first prove item~1 of the theorem.
Let $D$ be a compact (possibly branched) minimal disk with boundary
$\G $. Consider $D$ to be the image of a conformal
harmonic map $f\colon \D \to X$, where $\D $ is the closed
unit disk in $\C$ and $f|_{\partial \D}$ is
a homeomorphism to $\G$.
To prove that
$(\Pi\circ f)|_{\mbox{\small Int}(\D)} \colon \Int(\D) \to \Int(E)$ is
a diffeomorphism, we will modify a classical argument of
Rado~\cite{ra2} who proved a similar result for minimal
surfaces in $\rth$ whose boundaries have an orthogonal injective
projection to a convex planar curve.  Since $E$ is simply connected
and $(\Pi\circ f)|_{\mbox{\small Int}(\D)}\colon \Int(\D) \to \Int(E)$ is a
proper map, to prove item~1b it suffices to check that the
differential of $\Pi\circ f$ has rank two at every point of $\Int(\D)$.  By
contradiction, suppose that $p\in \Int(\D)$ is a point where the
differential of $\Pi\circ f$ has rank less than two.
In this case, either $f$ is unbranched at $p$ and the tangent plane $T_pD$ is vertical,
or $p$ is a branch point for $f$. We first consider the special case that
$f$ is unbranched at $p$ and the tangent plane $T_pD$ is vertical, or equivalently, there exists a
line $L\subset \R^2 \rtimes _A\{0\}$ passing through $(\Pi\circ f)(p)\in
\Int(E)$ such that the vertical plane $P=\Pi^{-1}(L)$ is tangent to
$D$ at the point $f(p)$.

The set $f^{-1}(P)\cap{\rm Int}(\D)$ is a
1-dimensional subset of $\D$ that contains no isolated points (at regular
points of $f$, this is a consequence of the maximum principle, while at branch
points of $f$ this follows from well known properties of branched minimal
surfaces). The $\Pi$-projection
of the boundary of $f[f^{-1}(P)]$ consists of two points in $L\cap C$ and
$f^{-1}(P)$ has locally around $p\in \Int (\D)$
the appearance of a system of at least two analytic segments crossing at $p$
(see e.g., Lemma~2 in Meeks and Yau~\cite{my1}).
Since $f(\Int(\D))$ is a proper analytic (possibly branched) surface in $\Int(W)$,
then we conclude that $f^{-1}(P)$ contains
the closure of the properly embedded analytic 1-complex
$f^{-1}(P)\cap \Int(\D)$. Furthermore, $f^{-1}(P\cap \G )$ consists of
two components, each of which is a closed interval (possibly a point;
this follows from the facts that $P\cap \G$ consists of two components,
$f(\Int (\D))=\Int (D)\subset \Int (W)$ and
$f|_{\partial \D}\colon \partial \D
\to \G $ is a homeomorphism), and
the limit set of $f^{-1}(P)\cap \Int(\D)$ intersects both of these components.
Note that each vertex in $f^{-1}(P)\cap \Int(\D)$
has a positive even number of
 associated edges with the number of edges at
 the vertex $p$ being at least 4.
As the component $\Delta $ of $f^{-1}(P)\cap \Int (\D )$
containing $p$ has at least 4 ends
or it contains  a simple closed curve,
then we conclude that either there is a simple closed curve
$\a$ in $\Delta $ or there is a
properly embedded arc $\a$ in $\Delta $
whose two ends are  contained
in the same component of $f^{-1}(P\cap \G)$;
in either case, there is a
compact subset $\D '$ of $\D$ with non-empty interior
bounded by $\a$ together with some connected subset of
$f^{-1}(P\cap \G)\subset \partial \D$.
Consider the product foliation $\{L_{t}\}_{t\in \R}$ of lines in
$\R^2\rtimes_A\{0\}$ parallel to $L=L_0$. Let $\{\Pi^{-1}(L_t)\} _t$
be the related foliation of $X$ by minimal vertical planes, and
without loss of generality, suppose that $f(\D') \cap [\cup_{t>0}
\Pi^{-1}(L_t)]\neq \mbox{\rm \O}$. Note that
$f(\parc \D')\subset \Pi^{-1} (L)$. By compactness of $\D'$, there is a
largest value $t_0>0$, such that $\Pi^{-1}(L_{t_0})\cap f(\D') \not =
\mbox{\rm \O}$. But at any point of this non-empty intersection, we
obtain a contradiction with the maximum principle applied to the
minimal surfaces $\Pi^{-1}(L_{t_0})$ and $f(\D')$. This contradiction
proves that if $p$ is not a branch point of $f$, then the differential of
$\Pi \circ f$ has rank two at $p$.

On the other hand, if $p$ is a branch point of $f$, then choose a horizontal line
$L\subset \R^2 \rtimes _A\{0\}$ through $\Pi(p)$ so that the
associated vertical plane $P=\Pi^{-1}(L)$ intersects $\G$ in exactly two points.
Then, the set $f^{-1}(P)\cap{\rm Int}(\D)$ is a
1-dimensional subset of $\D$ that contains no isolated points and
$f^{-1}(P)$ has locally around $p\in \Int (\D)$
the appearance of a system of at least two analytic segments crossing at $p$.
Arguing as in the previous case gives a contradiction, which proves that
$f$ cannot have interior branch points; therefore, item~1b holds. Clearly item~1b implies 1a and 1c.

Note that if $\G$ is of class $C^2$ and ${\Pi|_{\G }}\colon \G \to C$ a
$C^2$-immersion, then  by the statement and proof of
Lemma~\ref{lemma1}, $\Pi|_{D}$ has rank two at every point of
$\partial D$ and $D$ is an embedded disk transverse to $\partial W$
along $\G$. The remaining items of item~1 of the theorem follow directly from
item~1b and this rank-two property of $\Pi|_{ D}$ along $\partial
D$.

We next prove item~2 of the theorem.
By Theorem~1 in~\cite{my2}, there exists a  disk $D_L$ of finite
least area in $W$ with boundary $\G $ and every such
least-area disk  is an embedding.  By Lemma~\ref{lemma1}, the
interior of any compact branched minimal disk  in $X$ with $\partial
D=\Gamma$ must be contained in the interior of $W$ and so, any
least-area disk in $X$ with boundary $\G $ is contained in~$W$.  The existence of $D_L$
proves item~2a of
the theorem.

The existence of $D_T$ can be found by constructing barriers. First suppose that $\G$ is smooth.
Consider a compact branched minimal surface $\Sigma $ in $X$ with $\partial \Sigma =\G $.
By Lemma~\ref{lemma1}, $\Sigma \subset W$. Consider the closure $C_{\Sigma }$ of  the component of
$W-\Sigma$ that contains a representative of the top end of $W$, by which we mean  a closed region in $W$ above
some horizontal plane $\R^2\rtimes _A\{t_0\}$.   Then $\G$ is homotopically trivial in $C_{\Sigma }$
and  by the barrier results in~\cite{my2}, $\G$ is the boundary of a finite  least-area disk in $C_{\Sigma }$,
which must be embedded (in fact, the interior of such a disk is a $\Pi $-graph
over the interior of $E$ by item~1b of this theorem);
furthermore, any two such least-area disks in $C_{\Sigma }$ intersect only along their common boundary
$\G$. Since this collection of 'disjoint' least-area embedded disks in $C_{\Sigma}$ with boundary $\G$
forms a sequentially compact set (since they all have the same finite area in the homogeneously regular
manifold $X$) and these disks are ordered by the relative heights of their graphing functions,
then there exists a unique highest such least-area
disk above $\Sigma $ that we denote by $D_T(\Sigma )$.  Approximation results
in~\cite{my1,my2} imply that when $\G$ is only continuous, then there also
exists a similar  embedded highest least-area disk $D_T(\Sigma )$ in $C_{\Sigma}$.

We claim that all the least-area disks $D_T(\Sigma )$ defined in the last paragraph lie in a compact
region of $W$, independent of $\Sigma $. If trace$(A)=0$ (equivalently, $X$ is unimodular),
then $\mathcal{F}=\{ \R^2\rtimes _A\{ z\} \ | \ z\in \R \} $
is a minimal foliation of $X$, and a simple application of the maximum principle to any compact branched minimal
surface $\Sigma'$ in $X$ with boundary $\G $ and to the leaves of $\cF$ gives that $\max (z|_{\Sigma'})\leq
\max (z|_{\G })$, which proves the claim in this case. If $X$ is non-unimodular, we can assume after
scaling its metric that trace$(A)=2$. Suppose that the claim fails to hold. Then, there exists a sequence of least-area
disks $D_T(\Sigma _n){\subset W}$ associated to compact branched minimal surfaces $\Sigma _n$, with
$\partial \Sigma _n=\partial D_T(\Sigma _n)=\G $ for all $n$ and $\max (z|_{D_T(\Sigma _n)})\to \infty $ as $n\to \infty $
{(observe that the mean curvature comparison principle applied to $D_T(\Sigma _n)$ and to the leaves of
$\mathcal{F}$ ensures that $\min (z|_{D_T(\Sigma _n)})\geq \min (z|_{\G})$).}
Given $n\in \N$, let $p_n\in D_T(\Sigma _n)$ be a point where $z|_{D_T(\Sigma _n)}$ attains its maximum value.
By taking $n$ large enough, we can assume that the intrinsic distance from $p_n$ to $\G $ is greater than 2. As
the $D_T(\Sigma _n)$ are stable, they have uniform curvature estimates away from their boundaries;
in particular, the norm of the second fundamental form of $D_T(\Sigma _n)$ in the intrinsic ball of radius 1 centered
at $p_n$ is less than some positive constant $C$ independent of $n$. This implies that there exists
$\ve >0$ such that, for $n$ large enough, the intrinsic disk $D(p_n,\ve )$ of
radius $\ve $ in $\R^2\rtimes _A\{ z(p_n)\} $ centered at $p_n$
satisfies the following property:
\par
\vspace{.2cm}
(P)\ Every vertical line passing through a point in $D(p_n,\ve )$ intersects $D_T(\Sigma _n)$
near $p_n$.
\par
\vspace{.2cm}
\noindent
Property (P) follows from the following observation, whose proof we leave to the reader:
\par
\vspace{.2cm}
(O)\ Let $\Sigma_1,\Sigma_2$ be two smooth surfaces in a homogeneous 3-manifold $X$
 that are tangent at a common point~$p$, such that the intrinsic distance from $p$ to the
boundaries of these surfaces is at least 1.
If the norms of the second fundamental forms of these surfaces are less than some $C>0$,
 then there exists an $\ve=\ve(C)\in(0,1/2)$, less than the injectivity radius of $X$,
such that every point in the intrinsic ball $B_{\Sigma_1}(p,\ve)=\{ q\in \Sigma _1\ | \
d_{\Sigma _1}(p,q)<\ve \} $ (here $d_{\Sigma _1}$  denotes intrinsic distance in $\Sigma _1$)
is a normal graph over a subdomain of the corresponding intrinsic ball
$B_{\Sigma_2}(p,2\ve)$, of absolute value less than $\ve$.
\par
\vspace{.2cm}

\noindent
With property (P) at hand, we next find the desired contradiction that will prove our claim
in the case $X$ is non-unimodular.
Since the area element $dA_z$ for the restriction of the canonical metric
to the plane $\R^2\rtimes _A\{ z\} $ is $dA_z=e^{-2z}\, dx\wedge dy$, then we conclude that
$$
\text{Area}(E)\geq \text{Area}[\Pi (D(p_n,\ve ))]=e^{2z(p_n)}\text{Area}[D(p_n,\ve )]=\pi \ve ^2e^{2z(p_n)},
$$
which implies that $z(p_n)$ is bounded from above, a contradiction. Now our claim is proved.

Once we know that all the least-area disks $D_T(\Sigma )$ lie in a compact
region of $W$ independent of $\Sigma $, then we conclude that they also have
uniformly bounded area by the following argument:
consider the disk $D_{z_0}:= [\R^2\rtimes _A\{ z_0\} ]\cap \Pi^{-1}(E)$
where $z_0\gg 1$ is chosen sufficiently large so that
$D_T(\Sigma )$ lies under $D_{z_0}$ for every compact branched minimal surface $\Sigma $ in $X$ with boundary $\G $.
Consider the union $D'$ of  $D_{z_0}$ with the annular portion of $\Pi^{-1}(\G )$ below $D_{z_0}$ and above $\G $.
$D'$ clearly has finite area.  Since $D_T(\Sigma )$ has least area
among surfaces in the region of $W$ above $D_T(\Sigma )$ with boundary $\G $, then
the area of $D'$ is greater than or equal to the area of $D_T(\Sigma )$, as desired.

Given two of the disks,  $D_T(\Sigma_1)$, $D_T(\Sigma_2)$, then using their union as a barrier,
our previous arguments demonstrate that there is a least-area graphical disk $D'$ with boundary $\G$
that lies in the region of $\Pi^{-1}(E)$ above both of them;
here "above" means in the sense that the $\Pi $-graphing function $h\colon \Int(E)\to \R$ for $\Int(D')$
is greater than or equal to the graphing functions for the disks  $D_T(\Sigma_1)$, $D_T(\Sigma_2)$.
This notion of "above" induces a partial ordering on the set of disks of the form $D_T(\Sigma)$.
A standard compactness argument
using that the areas of these disks are uniformly bounded proves the existence of a minimal
disk $D_T$ with boundary $\G $  that is a maximal element in the partial ordering.
By item~1a of this theorem, $D_T$ is embedded,
and by construction, $D_T$ has least area among all compact surfaces in the closed
region of $W$ above $D_T$. This proves item~2b of the theorem. Item~2c about $D_B$ can be proven by
similar reasoning as in the proof of item~2b for $D_T$.

It remains to prove item~2d of the theorem.
Suppose that $M$ is a compact branched minimal surface in $X$ whose
boundary lies in the closed set $W(D_T,D_B)\subset W$
between the graphs $D_T$ and $D_B$. Note that $M\subset W$ by the arguments in the proof of Lemma~\ref{lemma1}.
Suppose that some point $p$ of $M$ lies in $W-W(D_T,D_B)$. First suppose that $p$ lies in the
closed  region
$D_T^+\subset W$ that lies above $D_T$.  Then using $D_T\cup (M\cap D_T^+)$ as a barrier, we obtain
a minimal disk $D_T'$ of least-area that lies above $D_T$, which contradicts that $D_T$ is the highest
disk that bounds $\G$.  This contradiction implies $M$ does not intersect the
interior of $D_T^+$; similar arguments imply that  $M$ does not intersect the
interior of the region of $W$ that lies below $D_B$.  Hence, the main statement of item~2d is proved.
Elementary separation arguments now imply
that $D_T$ and $D_B$ are uniquely defined, and clearly if $D_T=D_B$, then every compact branched minimal
surface in $X$ with boundary $\G$ is equal to $D_T$.
This completes the proof of Theorem~\ref{lemma1}.
\end{proof}

By item~2d of Theorem~\ref{lemma2},  if
a curve $\G$ satisfying the hypotheses of Theorem~\ref{lemma2}
is the boundary of a unique minimal disk, then it is also the boundary of
a unique compact branched minimal surface.  Our belief that
graphical minimal disks bounding such a $\G$ are unique lead us to the following conjecture.

\begin{conjecture} \label{conj}
 Let $X=\R^2\rtimes_A \R$ be a metric semidirect product.
Suppose $E$ is a compact convex
disk in $\R^2\rtimes_A\{0\}$, $C=\partial E$ and $\G\subset
\Pi^{-1}(C)$ is a continuous simple closed curve such that
$\Pi\colon \Gamma \to C$ monotonically parameterizes $C$.
Then  the compact embedded disk $D_L$ of finite least area
given in Theorem~\ref{Plateau:unique} is the unique compact branched minimal surface in $X$ with boundary
$\G$.
\end{conjecture}

The uniqueness property stated in Conjecture~\ref{conj} is clear in the
particular case that $X=\R^2\rtimes_A\R$ is unimodular and $\Gamma$ is
contained in the plane $\R^2\rtimes_A\{0\}$, by the maximum principle applied
to $D_L=E$ and to the foliation of minimal planes $\{ \R^2\rtimes_A\{z\} \ | \ z\in \R \}$.

Next we prove the following particular case of Conjecture~\ref{conj} for the case that $X$ is non-unimodular.%

\begin{proposition}
\label{circles1}
Let $X= \R^2\rtimes_A\R$ be a non-unimodular
metric semidirect product, and let $F_3(x,y,z)=
F_3^H+\partial _z$ be the right invariant vector field in $X$ given by \eqref{eq:6}.
Suppose that $E\subset \R^2\rtimes_A\{0\}$ is a compact convex disk
with $C^2$ boundary $\G$ that is {\em almost-transverse to $F_3^H$,} in the sense that
the inner product of the outward pointing unit conormal to $E$ along $\G$ with $F_3^H$ is greater than or
equal to zero.
Then, $\G$ is the boundary of a unique compact branched minimal surface which must therefore be the least-area,
embedded minimal disk $D_L$ given in
Theorem~\ref{Plateau:unique}.
\end{proposition}
\begin{proof}
Arguing by contradiction,
suppose that the proposition fails to hold. Then item~2 of Theorem~\ref{Plateau:unique}
implies that the embedded minimal disks $D_T,D_B$ described there with boundary
$\G$ are not equal.
Let $\eta_T,\eta_B$ denote the respective outward pointing unit conormals
to these minimal disks along $\G $.  Since
$F_3$ is a Killing vector field and $D_T,D_B$ are minimal, then the divergence theorem
gives
\begin{equation} \label{eqFLUX}
0=\int _{D_i}\div _{D_i}(F_3^{T_i})=\int_\G \langle \eta_i, F_3\rangle
=\int_\G \langle \eta_i, F_3^H\rangle+\int_\G \langle \eta_i, \partial_z\rangle,
\end{equation}
where $i=T,B$ and $\div _{D_i}(F_3^{T_i})$ denotes the intrinsic divergence in $D_i$
of the tangential component $F_3^{T_i}$ of $F_3$ to $D_i$.

 On the other hand, observe that by the boundary maximum principle, $D_B$ lies strictly
 below $D_T$ near their common boundary $\G $.

As $X$ is non-unimodular, then $\R^2\rtimes _A\{ 0\} $ has mean curvature 1, and so
both $D_T,D_B$ lie in $\R^2\rtimes_A [0,\infty)$
(this follows from the interior maximum principle applied
to $D_i$, $i=T,B$, and to $\R^2\rtimes_A \{ z_0\} $ for a suitable $z_0<0$)
and $D_T,D_B$ are transverse to $\R^2\rtimes_A\{0\}$ along $\G $ (by the
boundary maximum principle applied to $D_i$, $i=T,B$, and to $\R^2\rtimes _A\{ 0\} $).
Therefore, we deduce that
\begin{equation}
 \label{eqFLUX2}
 \langle \eta_T, \partial_z\rangle
< \langle \eta_B, \partial_z\rangle<0.
\end{equation}
Expressing $\eta _T$ as a sum of its horizontal and vertical components, we have
\[
\eta _T=\langle \eta _T,\eta _H\rangle \eta _H+\langle \eta _T,E_3\rangle E_3,
\]
where $\eta _H$ is the outward pointing unit conormal to $E$ along $\G $.
As $\langle F_3^H,E_3\rangle =0$, then
\[
\langle \eta _T,F_3^H\rangle =\langle \eta _T,\eta _H\rangle \ \langle \eta _H,F_3^H\rangle .
\]
Note that $\langle \eta _H,F_3^H\rangle \geq 0$ since $\G $ is almost transverse to $F_3^H$,
and that $\langle \eta _T,\eta _H\rangle \geq 0$ as $D_T$ lies in $\Pi^{-1}(E)$. Thus
$\langle \eta _T,F_3^H\rangle \geq 0$. Arguing similarly with $D_B$ we will obtain
$\langle \eta _B,F_3^H\rangle =\langle \eta _B,\eta _H\rangle \ \langle \eta _H,F_3^H\rangle $.
As $D_B$ lies below $D_T$ near $\G $, then $\langle \eta _T,\eta _H\rangle
\leq \langle \eta _B,\eta _H\rangle $. Altogether we deduce that
\begin{equation}
\label{eqFLUX3}
0\leq\langle \eta_T, F_3^H\rangle \leq \langle \eta_B, F_3^H\rangle.
\end{equation}
The inequalities \eqref{eqFLUX2} and \eqref{eqFLUX3} imply
\[
\int_\G \langle \eta_T, F_3^H\rangle+\int_\G \langle \eta_T, \partial_z\rangle
<\int_\G \langle \eta_B, F_3^H\rangle+\int_\G \langle \eta_B, \partial_z\rangle,
\]
which contradicts \eqref{eqFLUX}.  This contradiction proves the proposition.
%
\end{proof}

\section{Asymptotic  behavior of certain compact minimal surfaces in
  non-unimodular metric Lie groups.}
  \label{sec:4}

If $I_2$ is the identity matrix in $\mathcal{M}_2(\R )$, then the metric
Lie group  $X=\R^2\rtimes_{I_2} \R$ is isometric to hyperbolic 3-space, and the
planes $\R^2\rtimes_A\{t_0\}$ correspond to a family of horospheres with the same
point at the ideal boundary of $X$. In this case,
every  circle $\G _R$ of Euclidean radius $R>0$
in  the (flat) plane $\R^2 \rtimes_{I_2}\{0\} $ is the
boundary of a least-area compact disk
$\Sigma_R\subset \R^2 \rtimes_{I_2}[0, \infty) $,
which is contained in a totally geodesic hyperbolic plane $H$ in
$X$; in fact, $\Sigma _R$ is a geodesic disk of some hyperbolic radius in $H$.
  Furthermore, as the Euclidean radius of $\G _R$ in $\R^2 \rtimes_{I_2}\{0\}$ goes to infinity,
then the Riemannian radius of $\Sigma _R$ also goes to infinity and the constant angle that
$\Sigma _R$ makes with
$\R^2 \rtimes_{I_2}\{0\} $ approaches $\pi/2$, see Figure~\ref{fig1}.
\begin{figure}
\begin{center}
\includegraphics[height=6cm]{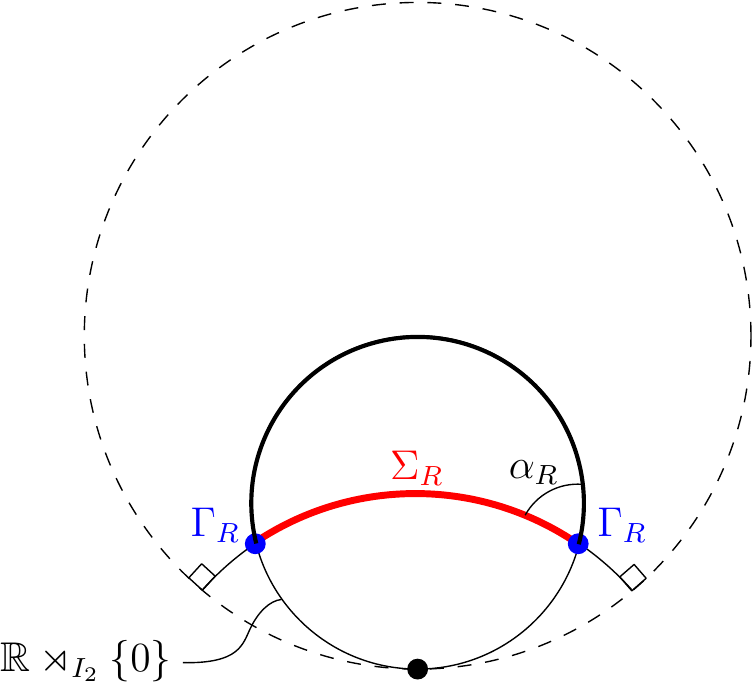}
\caption{The constant angle $\a _R$ that the least-area disk $\Sigma _R$ makes with the horosphere
$\R^2\rtimes _{I_2}\times \{ 0\}$ along $\partial \Sigma _R=\G_R$ approaches $\pi /2$ as the
Euclidean radius $R>0$ of the circle $\G _R$ in the flat plane $\R^2\rtimes _{I_2}\{ 0\}$ tends to infinity.
{Here we are using the Poincar\'e ball model of $\Hip ^3$, which is}
{isometric to $\R^2\rtimes _{I_2}\R $.}}
\label{fig1}
\end{center}
\end{figure}
The next theorem
shows that a similar result holds in any non-unimodular metric Lie group, since every such
a non-unimodular metric Lie group is isomorphic and isometric to a non-unimodular semidirect product.

\begin{theorem}
 \label{thm:Plateau}
Let $X=\R^2\rtimes_A \R$ be a metric semidirect product, where  $A\in\mathcal{M}_2(\R )$
satisfies equation \eqref{Axieta}.
\begin{enumerate}[1.]
\item For each $\ve>0$, there exists a $\de>0$ such that the following
property holds. Given a $C^2$ simple closed convex curve $\G \subset
\R^2 \rtimes_A \{0\}$ with geodesic curvature in $\R^2 \rtimes_A \{
0\} $ less than $\de$ in absolute value, then every compact branched minimal surface
$M\subset X$ with $\partial  M=\G$ satisfies that
\begin{enumerate}[(a)]
\item $M\subset \R^2\rtimes _A[0,\infty )$.
\item $\Int(M)\subset \Pi ^{-1}(\Int (E))\cap [\R^2\rtimes
_A(0,\infty)]$, where $E$ is the convex disk in $\R^2\rtimes_A \{0\}$ bounded by $\G $.
\item $\langle \eta,\partial_z\rangle <-1+\ve$ along $\G $, where $\eta $ is the
exterior unit conormal vector to $M$ along its boundary.
\end{enumerate}
\item Suppose that $\G(n)\subset \R^2 \rtimes_A \{0\}$  is a
sequence of $C^2$ simple closed convex curves  with $\vec{0}=(0,0,0)\in \G(n) $
such that the geodesic curvatures of  $\G (n)$ converge uniformly to $0$
and the curves $\G (n)$ converge on compact subsets to a line $L$ with $\vec{0}\in
L$ as $n \to \infty$. Then, for any sequence $M(n)$ of compact branched minimal disks
(or compact stable minimal surfaces)
with $\partial M(n)=\G(n)$,  the surfaces $M(n)$ converge $C^2$ on compact
subsets of $X$ to the vertical halfplane
$\Pi^{-1}(L)\cap[\R^2 \rtimes_A [0,\infty)]$, as $n\to \infty $.
\end{enumerate}
\end{theorem}
\begin{proof}
We will first prove a key assertion which, together with
properties of convex simple closed curves in $\R^2$, will lead to the proof of the
theorem.

\begin{assertion}
\label{ass11.6}
Given $\ve \in (0,\frac{\pi }{2})$, there exists $R=R(\ve )>0$ such that for any $r>R$
the following property holds. Let $C_r\subset \R^2\rtimes _A\{ 0\} $ be a circle of Euclidean radius
$r>0$ and let $D_B(r)$ be the minimal disk given by item 2c of Theorem~\ref{lemma2} with boundary $C_r$.
Then, the angle that the tangent plane to
$D_B(r)$ makes with $\R^2\rtimes _A\{ 0\}$ is greater than $\frac{\pi }{2}-\ve $ at every
point of $C_r$.
\end{assertion}
\begin{proof}[Proof of the Assertion]
Arguing by contradiction, suppose there exists $\ve \in (0,\frac{\pi }{2})$, a sequence of
circles $C_{r(n)}\subset \R^2\rtimes _A\{ 0\}$ with Euclidean radii $r(n)\nearrow \infty $ and points
$p_n\in C_{r(n)}$ such that the angle that the tangent plane to
$D_B(r(n))$ at $p_n$ makes with $\R^2\rtimes _A\{ 0\}$ lies in $(0,\frac{\pi }{2}-\ve ]$.
After left translating these disks we can assume that $p_n=\vec{0}$, and after choosing
a subsequence, the $C_{r(n)}$ converge as $n\to \infty $ to a line $L\subset \R^2\rtimes _A\{ 0\} $ and
the closed disks $E_{r(n)}\subset \R^2\rtimes _A\{ 0\} $ bounded by $C_{r(n)}$ converge to
one of the two closed halfplanes bounded by $L$, which we denote by  $L^+$.

By item 2c of Theorem~\ref{lemma2}, given $n\in \N$ the unique compact embedded, stable minimal disk
$D_B(r(n))$ associated to the circle $C_{r(n)}$ is a $\Pi $-graph
over $E_{r(n)}$, and item 2d of the same theorem implies the following property:
\par
\vspace{.2cm}
{\it
$(\star )$ Let $W_n$ be the closure of the non-compact complement of
$D_B(r(n))\cup [\R^2\rtimes _A\{ 0\} ]$ in $\R^2\rtimes _A[0,\infty )$.
If $M\subset X$ is a compact, connected branched minimal surface whose boundary lies in $W_n$, then $M\subset W_n$.
}
\par
\vspace{.2cm}
To see why property $(\star)$ above holds, we only need to observe the
following: let $D_T(r(n))$ denote the compact embedded stable minimal disk
with boundary $C_{r(n)}$ given in item 2b of Theorem~\ref{lemma2}. Note that
$D_T(r(n))$ lies in $W_n$. Hence, if $M$ is not contained in $W_n$,
there must exist a compact piece $M'$ of $M$ such that $\parc M'$ lies
in the region of $\R^2\rtimes_A [0,\8)$ bounded by $D_T(r(n))$ and
$D_B(r(n))$, but $M'$ has points outside that region. This would contradict
item 2d of Theorem~\ref{lemma2}, which proves property $(\star)$.

Since the circles $C_{r(n)}$ converge on compact subsets to the horizontal
straight line $L$ as $n\to \infty $, then a standard argument shows that after
choosing a subsequence, the $D_B(r(n))-C_{r(n)}$ converge to a minimal lamination
${\mathcal L}$
of $X-L$ (local uniform curvature estimates for the $D_B(r(n))$
hold because these minimal surfaces are stable, see~\cite{ros9,sc3}). ${\mathcal L}$ is contained in
$\Pi ^{-1}(L^+)\cap [\R^2\rtimes _A[0,\infty )]$, as $D_B(r(n))-C_{r(n)}\subset
\Pi ^{-1}(E_{r(n)})\cap [\R^2\rtimes _A[0,\infty )]$ for every $n$ and
the $E_{r(n)}$ converge to $L^+$ as $n\to \infty $. It is clear that $L$ is contained in the
closure of ${\mathcal L}$. {\bf We claim that $\overline{\mathcal L}\cap L^+=L$}: if not, then there exists
a point $p\in
(\overline{\mathcal L}\cap L^+)-L$, such that $p$ is the limit of a sequence of points
in $D_B(r(n))$. Consider a circle $C(p)$ of radius $\frac{1}{2}d(p,L)$ ($d$ denotes Euclidean distance),
and let $D_B(p)$ be the compact embedded, stable minimal disk with boundary $C(p)$ given by
item 2c of Theorem~\ref{lemma2}. A straightforward adaptation of Property $(\star )$ shows that
for $n$ sufficiently large, $D_B(r(n))$ lies in the
closure of the non-compact complement of $D_B(p)\cup [\R^2\rtimes _A\{ 0\} ]$ in $\R^2\rtimes _A[0,\infty )$.
This contradicts that $p$ is the limit of a sequence of points in $D_B(r(n))$, thereby proving our claim.

Since the $D_B(r(n))$ are all $\Pi $-graphs, then by standard arguments the leaves of
 ${\mathcal L}$  are either $\Pi $-graphs over pairwise disjoint domains in $L^+\subset \R^2\rtimes
_A\{ 0\} $, or they are contained in vertical half-planes. In particular, there exists a unique leaf of $\mathcal{L}$
whose closure contains $L$.
Since the disks $D_B(r(n))$ have uniform local ambient area bounds nearby their boundaries (this follows from
the fact that $D_B(r(n))$ is least area in a certain region of $\R^2\rtimes _A\R $, see the proof of Theorem~\ref{lemma2}),
then after choosing a subsequence the surfaces with boundary $D_B(r(n))$ converge smoothly to a surface with
boundary near $L$; clearly, the interior of this limit surface is included in the unique leaf of $\mathcal{L}$
that has $L$ in its closure. Let $D_{\infty }$ be the closure in $\R^2\rtimes _A\R $ of such a leaf. The above
arguments show that $D_{\infty }$ is a smooth $\Pi $-graph with boundary $L$ (smoothness of $D_{\infty }$
holds up to its boundary). Observe that, by construction, the angle that the tangent plane to
$D_{\infty }$ at $\vec{0}$ makes with $\R^2\rtimes _A\{ 0\}$ lies in $(0,\frac{\pi }{2}-\ve ]$.

Consider the right invariant vector field $V$ on $X$ such that
$V(\vec{0})$ is unitary 
 and tangent to~$L$. {\bf We claim that
$V$ is everywhere tangent to $D_{\infty }$.} Given $\tau \in L$,
consider the angle $\t (\tau )\geq 0$ that the $\Pi $-graph $D_\infty$ makes with the
plane $ \R^2\rtimes _A\{ 0\} $ at the point $\tau $, viewed as a function $\t \colon L\to [0,\pi/2]$.
Recall that $\t (\vec{0})\in(0,\frac{\pi }{2}-\ve ]$ and observe that $\tau \in L\mapsto \t (\tau )$ is analytic,
as both $X$ and $D_{\infty }$ are analytic.  If $\t =\t (\tau )$ is constant along $L$,
then the unique continuation property of elliptic PDEs implies that for any $\tau\in L$,  the left translation
$\tau* D_{\infty}$ of $D_{\infty }$ by $\tau $ is equal to $D_\infty$, which implies our claim.
Hence for the remainder of the proof of our claim, we will assume that there is
a $\tau\in L-\{ \vec{0}\}$ such that $\t (\vec{0})\neq \t(\tau )$.
\begin{enumerate}[(A)]
\item Suppose that $\t (\vec{0})>\t (\tau )$.
For $n\in \N$, consider a point $\sigma_n$ of
$\Int(L^+) $ at Euclidean distance
$\frac1n$ from $\tau $.  Then, the left translate $\sigma_n*D_B(r(n))$ of
$D_B(r(n))$ by $\sigma _n$ has boundary $\sigma_n*C_{r(n)}$. For $n\in \N $ fixed,
there exists an integer $j(n)>n$ such that for all $j\in \N$ with $j\geq j(n)$,
the boundary $\partial (\sigma _n*D_B(r(n)))=\sigma _n*C_{r(n)}$
lies in the interior of the disk $E_{r(j)}$. By Property $(\star )$ applied to $\sigma _n*D_B(r(n))$ and to
$M=D_B(r(j))$, we deduce that $D_B(r(j))$ lies in the closure $W_n$ of the non-compact complement of
$[\sigma _n*D_B(r(n))]\cup [\R^2\rtimes _A\{ 0\} ]$ in $\R^2\rtimes _A[0,\infty )$. Equivalently,
the graphing functions $v_n\colon \sigma _n*E_{r(n)}\to \R$, $u_{j}\colon E_{r(j)}\to \R $ such that
$\sigma _n*D_B(r(n))$ (resp. $D_B(r(j)$) is the $\Pi $-graph of $v_n$ (resp. of $u_{j}$),
satisfy $v_n\leq u_{j}$ in $\sigma _n*E_{r(n)}$, for every $j\geq j(n)$.
After taking limits as $j\to \infty $ (but letting $n$ fixed),
we conclude that $\sigma _n*D_B(r(n))$ lies below $D_\infty $. Taking limits as $n\to \infty $, we
have that $\tau*D_{\infty }$ lies below $D_\infty $. In particular, the angle that $\tau*D_{\infty}$ makes
with the plane  $\R^2\rtimes _A\{ 0\} $ at $\tau$ (which equals $\t (\vec{0})$) cannot be greater than
the angle that $D_{\infty}$ makes with $\R^2\rtimes _A\{ 0\} $ at $\tau$ (which is $\t (\tau )$), a contradiction
with our hypothesis in this case (A).

\item We next perform slight modifications in the arguments in (A) to find a contradiction
in the case that $\t (\vec{0})<\t (\tau )$. Given $n\in \N$ large, let $\sigma _n\in \mbox{Int}(E_{r(n)})$ be the point
at distance $1/n$ from $\vec{0}$ so that the segment $[\vec{0},\sigma _n]$ with end points $\vec{0}$ and $\sigma _n$ is orthogonal
to $C_{r(n)}$ at $\vec{0}$. Arguing as in case (A), there exists an integer $j(n)>n$ such that for every $j\in \N$,
$j\geq j(n)$, the left translate of $C_{r(n)}$ by $\sigma _n$ lies in the interior of $E_{r(j)}$.
By Property $(\star )$ applied to $\sigma _n*D_B(r(n))$ and to
$M=\tau * D_B(r(j))$, we deduce that $\tau * D_B(r(j))$ lies in the closure $W_n$ of the non-compact complement of
$[\sigma _n*D_B(r(n))]\cup [\R^2\rtimes _A\{ 0\} ]$ in $\R^2\rtimes _A[0,\infty )$. Equivalently,
the graphing functions $v_n\colon \sigma _n* E_{r(n)}\to \R$, $u_{j}\colon \tau * E_{r(j)}\to \R $ such that
$\sigma _n* D_B(r(n))$ (resp. $\tau * D_B(r(j)$) is the $\Pi $-graph of $v_n$ (resp. of $u_{j}$),
satisfy $v_n\leq u_{j}$ in $\sigma _n* E_{r(n)}$, for every $j\geq j(n)$.
Taking limits as $j\to \infty $ with $n$ fixed, we conclude that $\sigma _n* D_B(r(n))$ lies below $\tau * D_\infty $.
Taking limits as $n\to \infty $, we have that $D_{\infty }$ lies below $\tau * D_\infty $. In particular, the angle that
$D_{\infty}$ makes with the plane  $\R^2\rtimes _A\{ 0\} $ at $\tau$ (which equals $\t (\tau )$) cannot be greater than
the angle that $\tau* D_{\infty}$ makes with $\R^2\rtimes _A\{ 0\} $ at $\tau$ (which is $\t (\vec{0})$), that is,
$\t (\tau )\leq \t (\vec{0})$, which is contrary to our hypothesis.
\end{enumerate}

From (A) and (B) we conclude the proof of our claim
that $V$ is everywhere tangent to $D_{\infty}$.
\par
\vspace{.2cm}
Recall that $V$ is horizontal and right invariant. It follows from
equation~(\ref{eq:6}) that in $(x,y,z)$-coordinates in $X=\R^2\rtimes _A\R $,
$V$ is a linear combination of $\partial _x,\partial _y$ with constant
coefficients, and thus, its integral curves are horizontal lines, all parallel
to $L$ in the Euclidean sense. Since $V$ is everywhere tangent to $D_{\infty}$,
then the integral curves of $V$ passing through points in $D_{\infty }$ are
completely contained in $D_{\infty }$. Therefore, $D_{\infty }$ is
foliated by these horizontal lines. Let
$L^{\perp }\subset \R^2\rtimes _A\{ 0\} $ be the
straight line orthogonal to $L$ passing through
the origin. Since $D_{\infty }$ is a minimal
$\Pi $-graph and it is foliated by straight lines
parallel to $L$, then the intersection of $D_{\infty }$ with the
vertical plane $\Pi ^{-1}(L^{\perp })$ is a proper analytic arc $\g $
with $\vec{0}\in \g $, and $\g$ is a $\Pi $-graph over its projection to
$L^{\perp }$. Note that the $z$-coordinate restricted to $\g $
cannot have a local minimum value $z_0$, by the mean curvature comparison principle
applied to the minimal surface $D_{\infty }$ and to the mean curvature
one surface $\R^2\rtimes _A\{ z_0\} $. Since $\g $ is analytic, if $\g $ is not
parameterized by its $z$-coordinate, then $z|_{\g }$ must have a
first local maximum. As $\g $ lies above $\R^2\rtimes _A\{ 0\}$,
then $\g $ must be asymptotic to a horizontal line at height
$z_1\geq 0$  and
thus, $D_{\infty }$ is smoothly asymptotic to $\R^2 \rtimes
_A\{ z_1\} $. This contradicts that $D_{\infty }$ is minimal
and $\R^2\rtimes _A\{ z_1\} $ has mean curvature one. This
contradiction shows that $\g $ can be parameterized by its
$z$-coordinate; in fact, the range of values of $z$ along $\g $ is $[0,\infty)$
since $D_{\infty }$ cannot be smoothly asymptotic to any horizontal plane
$\R^2\rtimes _A\{ z_2\} $
for any $z_2>0$. In what follows we will parameterize $\g $ by the height
$z\in [0,\infty )$. 

To obtain the contradiction that will prove  Assertion~\ref{ass11.6}, we apply a flux argument to
an appropriate annular quotient of $D_{\infty }$. For any $t\in
(0,\infty)$, consider the minimal strip
\[
D_{\infty}(t)=D_{\infty}\cap (\R^2 \rtimes_A [0,t]).
\]
Fix $q\in L-\{ \vec{0}\}$ and consider
the infinite cyclic subgroup $\mathcal{I}$ of isometries of $X$ generated by
the left translation by $q$. Then $X/\mathcal{I}$ is a homogeneous
3-manifold diffeomorphic to $\esf^1\times \R^2$, the $z$-coordinate
on $X$ descends to a well-defined function on $X/\mathcal{I}$ (which we
will also call the height $z$), and every right invariant horizontal vector field
$F$ on $X$ descends to a well-defined Killing field $\wh{F}$
on $X/\mathcal{I}$. Consider the quotient minimal annulus $\Omega
(t)=D_{\infty}(t)/\mathcal{I}$ in $X/\mathcal{I}$.

First consider the case that the Milnor $D$-invariant of $X$ is positive.
We next prove that the length in $X/\mathcal{I}$ of the boundary curve $c_t$ of
$\Omega (t)$ at height $t$ converges to zero exponentially quickly
as $t\to \infty$. To see this, without loss of generality we may assume that $L$
points in the $x$-direction (after a rotation in the
$(x,y)$-plane, which does not change the ambient left invariant metric but does
change the matrix $A$). Then, $q=(x(q),0,0)$ with $x(q)\in \R -\{
0\} $ and equation~(\ref{eq:13}) gives that
\[
\mbox{length}(c_t)=\int _0^{x(q)}\| \partial _x\| \, dx=
\int_0^{x(q)}\sqrt{a_{11}(-t)^2+a_{21}(-t)^2}\, dx= \| \partial
_x\|(0,0,t)\, |x(q)|,
\]
where $(a_{ij}(z))_{i,j}$ denotes the matrix $e^{zA}\in
\mathcal{M}_2(\R )$. Now we deduce that length$(c_t)$ converges to zero
exponentially quickly as $t\to \in \infty $ from item~2a of
Proposition~\ref{propos2app} in the Appendix, as the Milnor $D$-invariant is positive.
The same item~2a of Proposition~\ref{propos2app} ensures that
all horizontal right invariant vector fields in $X$ have
lengths decaying exponentially as $z\to +\infty $.  Pick a right
invariant horizontal vector field $F$ in $X$ such that $F,V$ are linearly independent.
Let $\wh{F}$ be the quotient Killing field  of $F$ in $X/\mathcal{I}$.
Consider for each $t\geq 0$ the flux
of $\wh{F}$ across $c_t$, defined as
\[
\mbox{Flux}(\wh{F},c _t)=\int _0^{x(q)}\langle \wh{F},\eta \rangle dx,
\]
where $\eta $ is the unit vector field which is tangent to $D_{\infty }/\mathcal{I}$,
orthogonal to $c_t$ with $\langle \eta ,E_3\rangle \geq 0$.
   Since $\wh{F}$ has
constant length along $c_t$ with this constant being bounded as a function of $t>0$,
and the length of $c_t$ converge to zero as $t\to \infty $,
then $\mbox{Flux}(\wh{F},c _t)$ also converges to zero as $t\to
\infty$.  As $\mbox{Flux}(\wh{F},c_0)$ is non-zero and
$\mbox{Flux}(\wh{F},c_t)$ is independent of $t$ (because  the
divergence of the tangential component of $\wh{F}$ along
$\Omega (t)$ is zero), then we obtain a contradiction.
This contradiction completes the proof of Assertion~\ref{ass11.6}
in the case $D>0$.

If $D\leq 0$, then item~2b of Proposition~\ref{propos2app}
gives that $A$ is diagonalizable with one positive eigenvalue
$\l $ and another non-positive eigenvalue $\mu $. In this case, there
exists a
horizontal right invariant vector field $F$ on $X$
with $\| F(\vec{0})\| =1$, such that the norm of
$F$ decreases exponentially quickly as $z\to +\infty $ (namely, $F$ is
determined by $F(\vec{0})$ being the unitary eigenvector of the matrix $A$
associated to $\l $, since $\| F\| (z)=e^{-\l z}$ by item~2b
of Proposition~\ref{propos2app}). Let $\wh{F}$ be the quotient
Killing field of $F$ on $X/\mathcal{I}$.
Suppose for the moment that $F$ is not collinear with $V$.
As before, we may assume that $L$ points in the $x$-direction and
$q=(x(q),0,0)$. Then, the flux of $\wh{F}$ along $c_t$ satisfies
\[
|\mbox{Flux}(\wh{F},c_t)|\leq \int _0^{x(q)}\| F\| (x_0,y_0,t)\, dx.
\]
But the line element $dx$ grows at most exponentially as $z\to
+\infty $, and in fact at most as the function $e^{-\mu z}$.
Since $2=\mbox{trace}(A)=\l +\mu $,
then $\| F\| \, dx$ decays exponentially as $z\to +\infty $ and
thus, we arrive to a contradiction as in the former case $D>0$.

The last case we must consider is $D\leq 0$ and $F$ is proportional to $V$. In this case,
we still normalize $L$ to point in the direction of the $x$-axis,
and replace $F$ in the above computations by $\partial _y$. Then
$dx$ decays like $e^{-\l z}$ (with the same notation as before)
while $\| F\| $ increases at most like $e^{-\mu z}$  as $z\to +\infty $,
and the conclusion is the same. Now Assertion~\ref{ass11.6} is proved.
\end{proof}

We now prove item~1 in the statement of Theorem~\ref{thm:Plateau}.
Let $\G $ be a simple closed convex curve contained in $\R^2\rtimes _A\{ 0\}
$, and let $M\subset X$ be a compact branched minimal surface with
$\partial M=\G $. By Lemma~\ref{lemma1}, $\Int(M)\subset \Pi ^{-1}(\Int (E))$,
where $E$ is the convex disk in $\R^2\rtimes_A \{0\}$ bounded by $\G $.
If $\Int(M)$ is not contained in $\R^2\rtimes _A(0,\infty )$, then there exists a point
$p\in \Int (M)$ with smallest non-positive $z$-coordinate $z_0$.  An
application of the mean curvature comparison principle gives a
contradiction to the fact that the mean curvature of the plane $\R^2
\rtimes _A\{ z_0\}$ is $1$ and $M$ lies on the mean convex side of
this plane at the point $p$. Therefore, conditions $(a),(b)$ in item 1 of
Theorem~\ref{thm:Plateau} hold.

Next we prove that condition $(c)$ in item 1 holds.
Arguing by contradiction, suppose
that there exists an $\ve>0$,
a sequence $M(n)$ of compact branched minimal surfaces with
$C^2$ simple closed convex boundary curves $\G(n)\subset
\R^2\rtimes_A\{0\}$, and points $p_n\in \G (n)$ such that
\begin{enumerate}[(C1)]
\item The geodesic curvature of $\G (n)$ is uniformly going to zero as $n\to \infty $.
\item $ \langle \eta_n(p_n),  \partial_z(p_n)\rangle
 \geq -1+\ve$ for all $n\in \N$,
 where $\eta _n$ is the exterior unit conormal vector to $M(n)$ along $\G (n)$
 (observe that in order to make sense of $\eta _n$ we are using that $M(n)$ is an immersion
 near $\G(n)$ by item~2 of Lemma~\ref{lemma1}).
 \end{enumerate}

Let $R=R(\ve )$ be the positive number produced by Assertion~\ref{ass11.6} applied to
the value $\ve >0$ that appears in Condition~(C2) above. By Condition~(C1), for $n$ large
enough we can assume that the geodesic curvature of $\G(n) $ is less than $1/R$.
This property allows us to find a round disk
$\wh{E}_n\subset \R^2\rtimes_A \{0\}$ of radius $R$ that is contained in the disk
$E(n)$ bounded
by $\G(n)$ and such that $\wh{E}_n\cap E(n)= \{p_n\}$. Let $\wh{\G}(n)=\partial \wh{E}_n$
 and let $D_B(n)$ be the `lowest' minimal disk
bounded by $\wh{\G}(n)$ given by item~2c of Theorem~\ref{lemma2}.
Since $D_B(n)$ lies below the branched minimal
surface $M(n)$ (this follows from Property~$(\star)$ in the proof of Assertion~\ref{ass11.6}
applied to $D_B(n)$ and $M(n)$)
and $D_B(n)\cap M(n)=\{p_n\}$,
then the angle that the tangent plane
$T_{p_n}M(n)$ makes with $\R^2\rtimes_A\{0\}$ is
greater than the angle $\varphi _n$ that $T_{p_n}D_B(n)$ makes with
$\R^2\rtimes_A\{0\}$.  Since
Condition~(C1) and Assertion~\ref{ass11.6} allow us to take $\varphi _n$ arbitrarily
close to $\pi /2$ for $n$ sufficiently large, then we conclude that
the angle that $T_{p_n}M(n)$ makes with $\R^2\rtimes_A\{0\}$
converges to $\pi /2$ as $n\to \infty$.  This  contradiction completes the proof of item~1c
of the theorem.

We next prove item~2 of the theorem. Suppose that
$\G(n)\subset \R^2 \rtimes_A \{0\}$
is a sequence of $C^2$ simple closed convex
curves  with $\vec{0}\in \G(n) $, having geodesic curvatures uniformly
approaching $0$ as $n \to \infty$ and converging on compact subsets
to a straight line $L$ that contains $\vec{0}$. Let $M(n)$ be a sequence of
compact branched minimal disks (or compact stable minimal surfaces)
with $\partial M(n)=\G(n)$. Suppose for the moment that the $M(n)$
are disks; we will discuss later the changes necessary to
prove the case that the $M(n)$ are stable.
By item~1 of Theorem~\ref{lemma2},
the disks $M(n)$ are unbranched and $\Pi $-graphs over the
compact convex disks $E(n)$ bounded by $\G (n)$ in $\R^2 \rtimes_A
\{0\}$. We claim that the $M(n)$ have uniformly bounded second
fundamental forms up to their boundaries; to see this, suppose this property
fails. Left translate the $M(n)$ so that the norm of
the second fundamental form is largest
at the origin, and rescale the $(x,y,z)$-coordinates by this maximum norm of the
second fundamental form of the $M(n)$, obtaining a new sequence
of rescaled minimal $\Pi $-graphs with uniformly bounded second fundamental
form. After extracting a subsequence,
these rescaled $\Pi $-graphs converge to a non-flat minimal surface
$M_{\infty }$ in $\R^3$ possibly with boundary
(if $\partial M_{\infty }$ is non-empty then $\partial M_{\infty }$
is a horizontal straight line and $M_{\infty }$ lies entirely
above the horizontal plane that contains $\partial M_{\infty }$).
Note that the Gaussian image of $M_{\infty }$ is contained in the closed
upper hemisphere,
which is clearly impossible if some component of $M_{\infty }$ has empty
boundary (note that this component would be complete). This implies that $M_{\infty }$ is connected and has nonempty
boundary. It follows that $M_{\infty }$ is a graphical stable minimal surface in the
closed upper half-space of $\rth$ bounded by the horizontal plane that contains
$\partial M_{\infty}$, which
can also be easily ruled out, since $M_{\infty }$ together with its image under the
180$^o$-rotation around $\partial M_{\infty }$ is a complete, non-flat minimal graph.
Therefore, the $M(n)$ have uniformly bounded second
fundamental forms up to their boundaries.

It follows that a subsequence of the $M(n)$ (denoted in the same way)
converges as $n\to \infty $ on compact subsets of $X$ to a minimal lamination
$\mathcal{L}$ of $X-L$, and $\mathcal{L}$ contains a leaf $M_{\infty }$
with boundary the straight line $L$. Since the geodesic curvatures of the curves
$\G(n)$ converge uniformly to $0$ as $n\to \infty $, then item~1 of this theorem
implies that $\langle \eta _n,\partial _z\rangle $ is arbitrarily close to
$-1$ for $n$ large enough (here $\eta _n$ is the exterior conormal vector field
to $M(n)$ along $\G(n)$). It follows that the limit surface
$M_{\infty }$ is tangent to the closed vertical
halfplane $\Pi^{-1}(L)\cap [\R^2\rtimes _A[0,\infty )]$ along $L$.
Since the $M(n)$ all lie at one side of
$\Pi^{-1}(L)\cap [\R^2\rtimes _A[0,\infty )]$,
then $M_{\infty }$ also lies at one side of
$\Pi^{-1}(L)\cap [\R^2\rtimes _A[0,\infty )]$
and thus, the boundary maximum principle implies that
 $M_{\infty }=\Pi^{-1}(L)\cap [\R^2\rtimes _A[0,\infty )]$.
We now prove that $\mathcal{L}$ contains no other leaves
different from $M_{\infty }$. Arguing by contradiction, any other leaf
component $\Sigma $  of $\mathcal{L}$ must be a complete positive $\Pi $-graph (without boundary) over
its projection to $\R^2 \rtimes_A \{0\}$, and $\Sigma $ has bounded
second fundamental form by arguments in the previous paragraph. But the existence of such a graphical leaf
$\Sigma$ in $\R^2\rtimes _A(0,\infty )$ is easily seen to be
impossible by considering its behavior on a sequence of points
$p_k=(x_k,y_k,z_k)\in \Sigma $ where $\lim_{k \to \infty}z_k$ is the infimum
$z_0\geq 0$ of the $z$-coordinate function of $\Sigma$ (recall that the minimal
surface $\Sigma $ cannot be asymptotic to the mean curvature one
surface $\R^2\rtimes _A\{ z_0\}$). This contradiction proves that
$\mathcal{L}=\{ M_{\infty }\} $, and thus, a subsequence of the $M(n)$
converges to the desired halfplane. Since every subsequence of the
$M(n)$ has a convergent subsequence which equals this limit, then
the entire sequence $M(n)$ converges to
$\Pi^{-1}(L)\cap [\R^2\rtimes _A[0,\infty )]$.
This completes the proof of item~2 of the theorem
in the case that the $M(n)$ are disks.

If the $M(n)$ are compact stable surfaces (not disks)
then the curvature estimates by Schoen~\cite{sc3} and Ros~\cite{ros9} give that
the $M(n)$ have uniformly bounded second
fundamental forms away from their boundaries.
As previously, for each $n\in \N$ let $E(n)$ be the convex compact disk
bounded by $\G (n)$ in $\R^2 \rtimes_A\{0\}$.
Note that by barrier arguments as in the proof of items 2b, 2c in
Theorem~\ref{lemma2}, for each $n\in \N$  there exists
a least-area disk $D(n)$ with
boundary $\G (n)$  in the closure of
the bounded region of  $[\R^2\rtimes _A [0,\infty)] -M(n)$ that contains $E(n)$.
Furthermore, $D(n)\subset \Pi ^{-1}(E(n))
\cap [\R^2\rtimes _A[0,\infty )]$
 is a $\Pi $-graph over $E(n)$.
Also, $M(n)$ lies ``above'' the $\Pi $-graph $D(n)$. As
the previously considered case of disks ensures that the $D(n)$ converge to
$\Pi^{-1}(L)\cap [\R^2\rtimes _A[0,\infty )]$ as $n\to \infty $,
then the $M(n)$ converge (as sets) to
$\Pi^{-1}(L)\cap [\R^2\rtimes _A[0,\infty )]$.
We now check that the last convergence is of class $C^2$,
by showing that the $M(n)$ have uniformly
bounded second fundamental form up to their
boundaries (this would finish the proof of
item~2 of Theorem~\ref{thm:Plateau}
in this case). If this is not the case,
then the rescaling-by-curvature argument above produces a
limit of a subsequence of the $M(n)$ which is a
non-flat, stable minimal surface $M_{\infty }$ in $\R^3$,
such that either has no boundary, or $M_{\infty }$ has
non-empty boundary given by a horizontal
line and $M_{\infty }$ is contained in a quarter
of space $Q\subset \R^3$ with $\partial M_{\infty }$
being the set of non-smooth points of $\partial Q$.
If $\partial M_{\infty }= \mbox{\O }$, then $M_{\infty }$
is complete, which contradicts that $M_{\infty }$ is non-flat and stable.
Therefore, $\partial M_{\infty }\neq \mbox{\O }$. In this case,
the rescaled least-area disks $D(n)$ are all below
the related rescaled $M(n)$. Since these rescaled images of $D(n)$ converge
as $n\to \infty $ to a vertical half-plane in $Q$, then $M_{\infty }$ must be equal
to this vertical limit half-plane, which contradict the non-flatness of $M_{\infty }$.
Now the theorem is proved.
\end{proof}

\begin{corollary}
\label{corol5.4}
Let $X=\R^2\rtimes _A\R $ be a metric semidirect product, where $A\in \mathcal{M}_2(\R )$ satisfies equation (\ref{Axieta}).
Then, there exists a straight line $L\subset \R^2\rtimes _A\{ 0\} $ with $\vec{0}\in L$ such that the following property holds.
\begin{enumerate}[(Q)]
\item Let $p,q\in L$ be different points with $\vec{0}\in (p,q)$ (here $(p,q)\subset L$ is
the open segment with extrema $p$, $q$), and
let $C_p,C_q\subset \R^2\times _A\{ 0 \}$ be
pairwise disjoint Euclidean circles centered at points in $L-[p,q]$,
with $p\in C_p$, $q\in C_q$. If the Euclidean
radii of $C_p,C_q$ are sufficiently large, then
there exists an embedded least-area annulus
$\Sigma \subset \R^2\rtimes _A[0,\infty )$
with boundary $\partial \Sigma =C_p\cup C_q$ (see Figure~\ref{fig2}).
\end{enumerate}
Furthermore:
\begin{enumerate}[(1)]
\item If the Milnor $D$-invariant of $X$ is $D> 0$, then property (Q) holds
for every line $L\subset \R^2\rtimes _A\{ 0\} $ with $\vec{0}\in L$.
\item If $D\leq0$, then property (Q) holds for the line $L\subset \R^2\rtimes _A\{ 0\} $
with $\vec{0}\in L$ in the direction of the eigenvector of $A$ associated to a
positive eigenvalue.
\end{enumerate}
\end{corollary}
\begin{figure}
\begin{center}
\includegraphics[height=5.3cm]{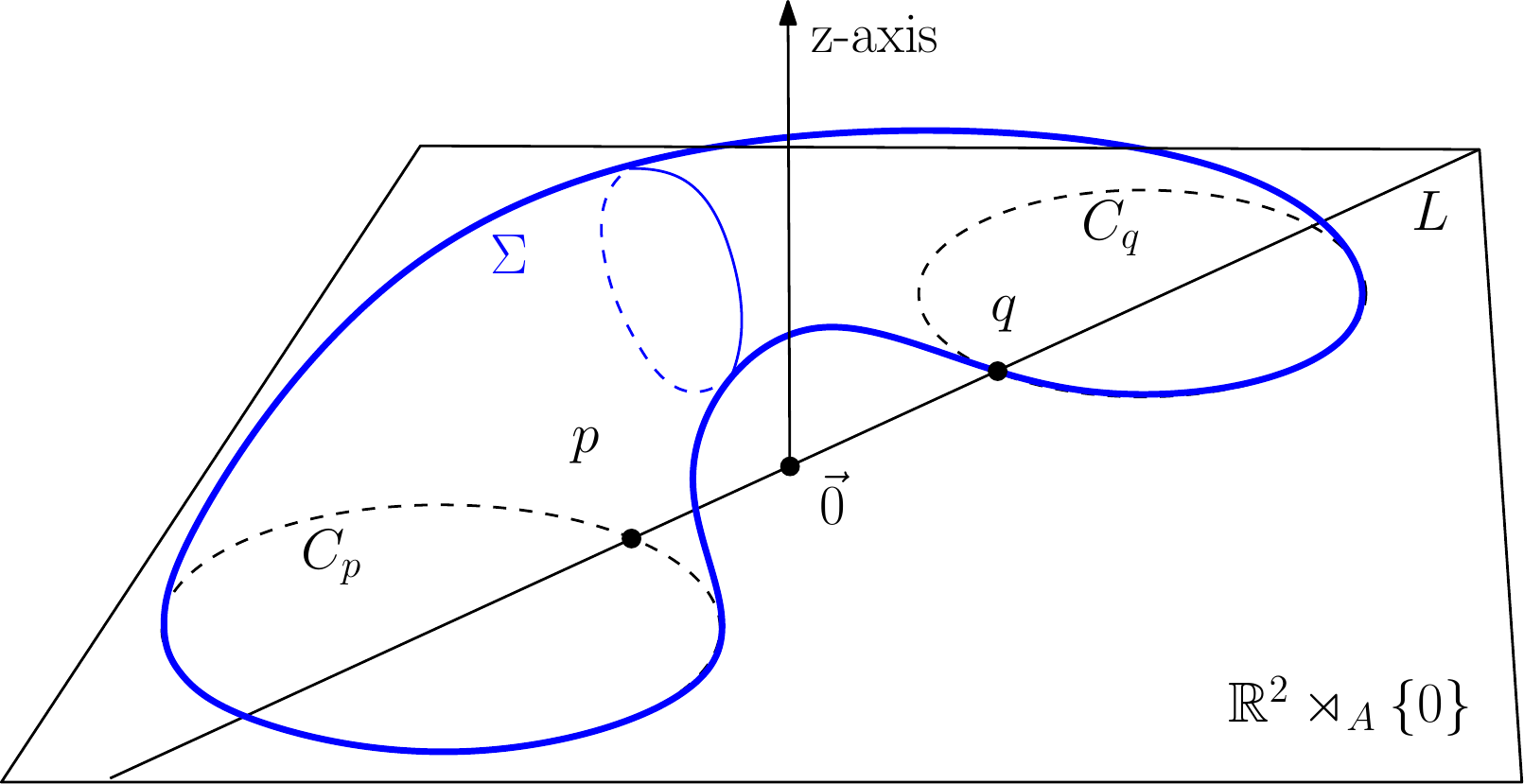}
\caption{The minimal annulus $\Sigma $ that appears in Corollary~\ref{corol5.4}, for
a pair of pairwise disjoint, large enough circles $C_p,C_q\subset \R^2\rtimes _A\{ 0\} $.}
\label{fig2}
\end{center}
\end{figure}
\begin{proof}
Suppose first that $D>0$, and let $L\subset \R^2\rtimes _A\{ 0\} $ be any line with $\vec{0}\in L$.
By item 2a in Proposition~\ref{propos2app},
the vertical lines $l_p=\{ (p,z)\ | \ z\in \R \} $ and $l_q=\{ (q,z)\ | \ z\in \R \} $
are both asymptotic to the $z$-axis as $z\to \infty $. Now consider
pairwise disjoint Euclidean circles $C_p,C_q\subset \R^2\times _A\{ 0 \}$ centered
at points in $L-[p,q]$, with $p\in C_p$, $q\in C_q$. Let $D_p,
D_q$ be compact, embedded, least-area disks with
boundaries $\partial D_p=C_p$, $\partial D_q=C_q$, which exist by Theorem~\ref{lemma2}.
By Assertion~\ref{ass11.6}, if the Euclidean radii of $C_p$, $C_q$ is sufficiently large,
then $D_p,D_q$ are arbitrarily close to the vertical half-planes
$\Pi ^{-1}(L_p)\cap [\R^2\rtimes _A[0,\infty )]$,
$\Pi ^{-1}(L_q)\cap [\R^2\rtimes _A[0,\infty )]$, where $L_p,L_q\subset \R^2\rtimes _A\{ 0\} $
are the lines orthogonal to $L$ that pass through $p,q$ respectively.
Since $l_p\subset L_p$ and $l_q\subset L_q$ are asymptotic to the $z$-axis as $z\to
\infty$, then $l_p$ and $l_q$ are asymptotic to each other, and thus the distance between the disjoint area minimizing disks $D_p$
and $D_q$ is arbitrarily small if the Euclidean radii of $C_p$, $C_q$ are sufficiently large. Therefore, after replacing
a pair of intrinsic geodesic disks $D'_p\subset D_p, D'_q\subset D_q$  of radius 1 centered at
sufficiently close points of $D_p,D_q$
by an annulus of least area with boundary $\partial D'_p\cup \partial D'_q$, we obtain a piecewise smooth annulus with area less
that the sum of the areas of the least-area disks $D'_p,D'_q$.  By the Douglas
criterion (the area of some annulus bounding $C_p\cup C_q$ is less than the infimum of the areas of any two disks
bounding $C_p\cup C_q$), there exists by Morrey~\cite{mor1} an annulus $\Sigma $ of least area in $X$
with boundary $C_p\cup C_q$.  Note that $\Int(\Sigma )\subset \R^2\rtimes[0,\infty)$ by the
maximum principle applied to $\Sigma $ and to planes $\R^2\rtimes _A\{ z\} $ with $z<0$,
then by the Geometric Dehn's Lemma for
Planar Domains given in Theorem~5 in~\cite{my1}, $\Sigma$ is a smooth embedded annulus (actually, Theorem~5
in~\cite{my1} is stated for three-manifolds with convex boundary but the convex boundary is only used to obtain 
the existence of a least-area immersed annulus, which we already have in this case).

Suppose $D\leq0$.
As the eigenvalues of $A$ are the roots of the polynomial $\l ^2-2\l +D=0$, then $\l =1\pm \sqrt{1-D}$. Hence,
exactly one of these eigenvalues $\lambda_+$ is greater than or equal to $2$, and the other one is non-positive.
After an orthogonal change of basis (that does not change the metric Lie group
structure of $X$), the matrix $A$ transforms to $A_1=OAO^{-1}$ for some
orthogonal matrix $O$,
where $A_1$ has entries $a_{11}=\lambda_+$ and $a_{21}=0$, and having an
 associated eigenvector $(1,0)$.
Consider the line $L=\{(t,0,0)\}\subset \R^2\rtimes _{A_1}\{ 0\} $.
In this case, equation~\eqref{eq:6*} shows that $E_1=e^{\lambda_+z}\partial_x$, and
so, $\| \partial _x\|\ $ is exponentially decaying as $z\to \infty $.
Hence, for any pair of different points $p,q\in L$,
the vertical lines $l_p=\{ (p,z)\ | \ z\in \R \} $ and $l_q=\{ (q,z)\ | \ z\in \R \} $
are both asymptotic to the $z$-axis as $z\to \infty $. Now consider
pairwise disjoint Euclidean circles $C_p,C_q\subset \R^2\times _{A_1}\{ 0 \}$ centered
at points in $L-[p,q]$, with $p\in C_p$, $q\in C_q$. Let $D_p,
D_q$ be compact, embedded, least-area disks with boundaries $\partial D_p=C_p$, $\partial D_q=C_q$,
which exist by Theorem~\ref{lemma2}.
Arguing as in the previous paragraph, if the Euclidean radii of $C_p$, $C_q$ are sufficiently large,
there exists an embedded least-area annulus $\Sigma \subset \R^2\rtimes _A[0,\infty )$
with boundary $\partial \Sigma =C_p\cup C_q$, and the proof is complete.
\end{proof}

\section{Radius estimates for cylindrically bounded stable minimal surfaces.}
\label{sec:estim}
In this section we obtain radius estimates for compact stable minimal surfaces in
semidirect products, using the results from Section~\ref{sec:4}.
Given $r>0$ and a vertical geodesic $\G$ in a metric semidirect product $X=\R^2\rtimes _A\R$,
we will denote by  $\mathcal{W}(\G ,r)\subset X$ the closed solid metric cylinder
of radius $r$ centered along $\G$.

\begin{proposition}
\label{prop5.1}
Let $X=\R^2\rtimes_A \R$ be a metric semidirect
product, where $A$ is either as in equation (\ref{Axieta}) with Milnor $D$-invariant less than
1 or where {\em trace}$(A)=0$ \,(this is the case where $X$ is unimodular).
For every vertical geodesic $\G \subset X$ and $r>0$, there exists a $j\in \N$
such that every compact immersed minimal surface $M$ in $X$ with
$\partial M\subset \mathcal{W}(\G ,r)$ satisfies $M\subset \mathcal{W}(\G ,jr)$.
\end{proposition}
\begin{proof}
Without loss of generality, we will henceforth assume that  $\G$ is the
$z$-axis. Recall that if $X$ is unimodular, then it is isomorphic to $\R^3$,
$\EE$, ${\rm Nil}_3$ or ${\rm Sol}_3$.

First suppose that $X$ is not isomorphic to $\EE $ or  Nil$_3$.
By Theorem~3.6 in~\cite{mpe11} (see also Examples 3.2--3.5 therein),
there are two distinct vertical planes $P_1,P_2$ (in $(x,y,z)$-coordinates, in fact,
$P_1$ can be taken as the $(x,z)$-plane and $P_2$ as the $(y,z)$-plane)
that are Lie subgroups of $X$.
For $i=1,2$, let $U_i(R)$  be the closed regular
neighborhood of $P_i$ of radius $R>0$. {\bf We claim that the boundary surfaces
$\partial U_i(R)=P_i^+(R)\cup P_i^-(R)$ of $U_i(R)$
both have non-negative
mean curvature  in $X$ with respect to the inward pointing normal
to $U_i(R)$.} Since each of $P_i^+(R),P_i^-(R)$
are at constant distance from $P_i$, which is a connected, codimension-one subgroup in
$X$, then Lemma~3.9 in~\cite{mpe11} implies that $P_i^{\pm }(R)$
is a right coset of $P_i$ and is also a left coset of some 2-dimensional subgroup
$\Sigma _i^{\pm }(R)$ of $X$.
This last property implies that $P_i^{\pm }(R)$ has constant mean curvature,
as every 2-dimensional subgroup in a metric Lie group has this property. Hence it remains to
show that the mean curvature vector of $P_i^{\pm }(R)$ points towards $U_i(R)$.
Note that $\Sigma _i^{\pm }(R)$ must
be disjoint from $P_i^{\pm }(R)$ (otherwise $\Sigma _i^{\pm }(R)=P_i^{\pm }(R)$,
which implies $\vec{0}\in P_i^{\pm }(R)\cap P_i$ hence $R=0$, a contradiction).
In this situation, the classification of codimension-one subgroups
in Theorem~3.6 in~\cite{mpe11} implies that $\Sigma _i^{\pm }(R)$ is one of
the elements in the 1-parameter family $\mathcal{A}_i$
of 2-dimensional subgroups of $X$ that share the
1-dimensional subgroup $P_i\cap [\R^2\rtimes _A\{ 0\} ]$ (also called an
{\it algebraic open book decomposition} of $X$).
In the case that $X$ is unimodular (hence isomorphic to $\R^3$ or $\sol$), then
item~6 of Theorem~3.6
in~\cite{mpe11} implies that all 2-dimensional subgroups of $X$
are minimal, hence $P_i^{\pm }(R)$ are minimal surfaces  and the claim is proved in this case.
Next we will prove the desired mean convexity of $U_i(R)$ in the case that $X$ is
non-unimodular and for $i=1$ (for $i=2$
the argument is similar and we leave it for the reader).
Item~5 of Theorem~3.6 in~\cite{mpe11} ensures that
up to possibly rescaling the metric, $X$ is isometric and isomorphic to $\R^2\rtimes _{A(b)}
\R $ for a diagonal matrix of the form $A(b)=\left(
\begin{array}{cc}
1 & 0  \\
0 & b
\end{array}
\right) $, for some $b\in \R$, $b\neq -1$; furthermore, we can assume that $P_1=\{ y=0\} $ and thus,
the algebraic open book decomposition of $\R^2\rtimes _{A(b)}\R $ that contains $P_1$
as one of its leaves is $\mathcal{A}_1=\{ H_1(\l )\ | \ \l \in \R \cup \{ \infty \} \} $,
where
\[
H_1(\l )=\left\{ \begin{array}{ll}
\{ (x,\frac{\l }{b}(e^{bz}-1),z)\ | \ x,z\in \R \} & \mbox{ if $b\neq 0, \l \in \R $,}
\\
\{ (x,\l z,z)\ | \ x,z\in \R \} & \mbox{ if $b=0, \l \in \R $,}
\\
\R^2\rtimes _{A(b)} \{ 0\} & \mbox{ if $\l =\infty $;}
\end{array}
\right.
\]
(hence $P_1=H_1(0)$). Observe that the 2-dimensional subgroups in $\mathcal{A}_1$
are products with the $x$-factor of proper graphs of the $z$-variable in the $(y,z)$-plane;
this applies in particular
to $P_1$ and to $\Sigma _1^{\pm }(R)$. Therefore,
$\partial _x$ is everywhere tangent to $P_1$ and to $\Sigma _1^{\pm }(R)$. As $P_1^{\pm }(R)$ is a
right coset of $P_1$ and $F_1=\partial _x$ is a right invariant vector field,
then $\partial _x$ is also everywhere tangent to $P_1^{\pm }(R)$. In other words,
$P_1^{\pm }(R)$ is the product with the $x$-factor of a curve in the $(y,z)$-plane.
In fact, this curve must be a proper graph of the $z$-variable (to see this, observe that
every horizontal plane $\R^2\rtimes _{A(b)}\{ z\} $ intersects $P_1^{\pm }(R)$ in a
line parallel to the $x$-axis which is the set of points of $\R^2\rtimes _{A(b)}\{ z\}$ at
distance $R$ from $P_1$). Now the
desired mean convexity of $U_1(R)$ with respect to the inward normal vector can be
understood by considering the related problem for $z$-graphs in the $(y,z)$-plane (i.e.,
after taking quotients in the $x$-factor): to do this, observe that $P_1^+(R)$ lies entirely at
one side of $P_1$, say its right side, see Figure~\ref{fig3}.
\begin{figure}
\begin{center}
\includegraphics[height=9cm]{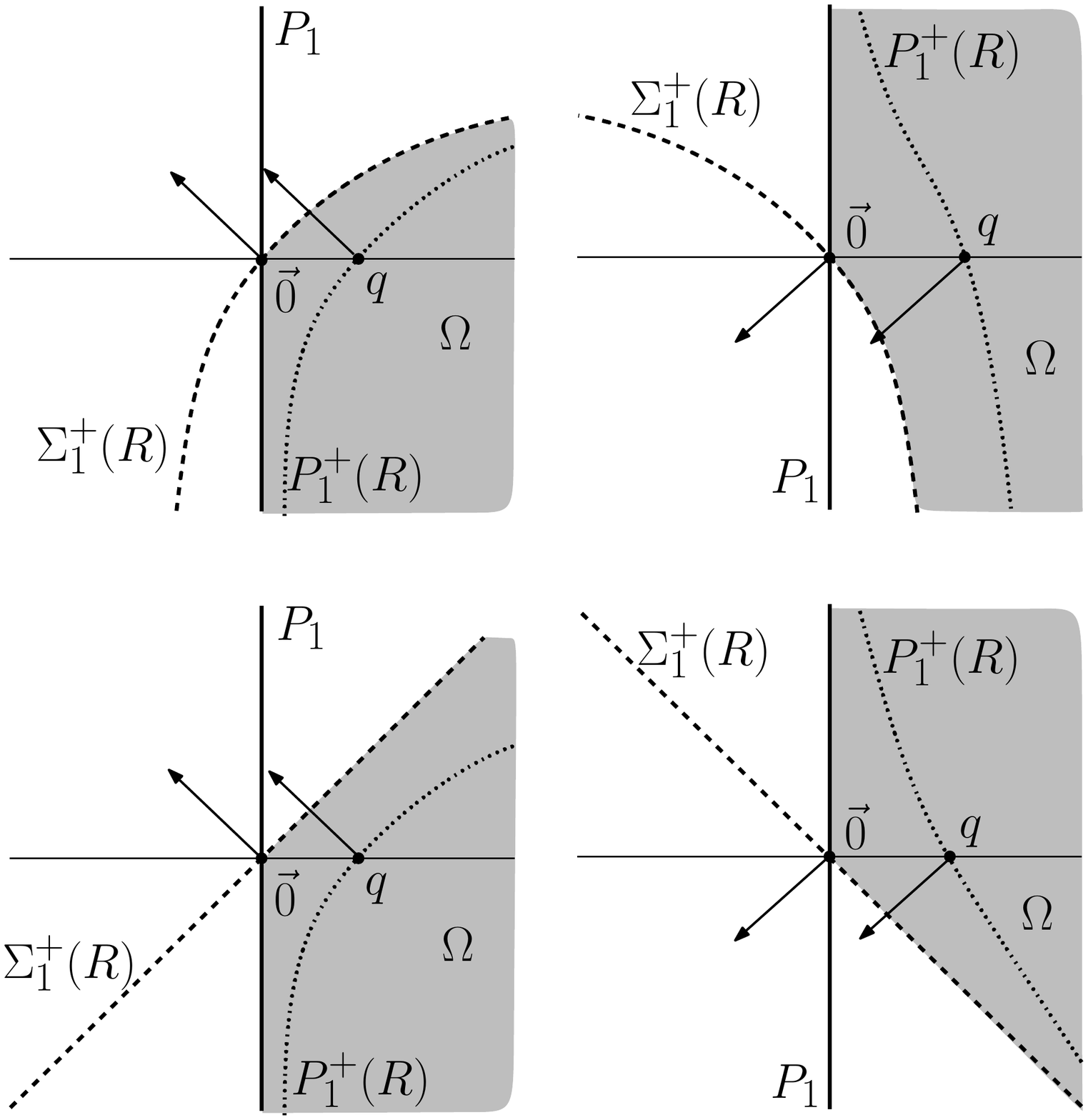}
\caption{The mean curvature vector of $P_1^+(R)={q * \Sigma _1^+(R)}$ (dotted) points
towards the 2-dimensional subgroup $\Sigma _1^+(R)$ (dashed), and hence towards the
vertical plane $P_1$. Above:
the case $b\neq 0$. Below: the case $b=0$. All graphics are representations in
the $(y,z)$-plane.}
\label{fig3}
\end{center}
\end{figure}
We can write $P_1^+(R)=q{*\Sigma _1^+(R)}$, the left coset of $\Sigma _1^+(R)$
obtained after left multiplication by an element $q\in  P_1^+(R)$.
Observe that if $\Sigma _1^+(R)=P_1$ then $\Sigma _1^+(R)$ is minimal
(see Remark~\ref{rem3.2}), and hence $P_1^{\pm }(R)$ is minimal as well,
which gives the desired mean convexity in this case.
Thus, we can assume that $\Sigma _1^+(R)\neq P_1$.
As $P_1^+(R)$ lies at the right side of $P_1$ and is disjoint from $\Sigma _1^+(R)$, then
$P_1^+(R)$ lies in the component $\Omega $ of $[\R^2\rtimes _{A(b)}\R ]-
[P_1\cup \Sigma _1^+(R)]$ that contains $[\R^2\rtimes _{A(b)}\{ 0\} ]\cap \{ y>0\} $
(see Figure~\ref{fig3}).
As the mean curvature vector of $P_1^+(R)$ at $q$ equals the mean curvature vector
of $\Sigma _1^+(R)$ at $\vec{0}$, then a continuity argument in the variable $R$ gives that
the mean curvature vector of $P_1^+(R)$ at $q=q(R)$ points towards $P_1$,
which finishes the proof of the claim.

By the last claim, the maximum principle (for the case $X$ is unimodular) and
the mean curvature comparison principle (when $X$ is non-unimodular)
applied to the foliation $\{ P_i^-(R)\ | \ R>0\} \cup
\{ P_i\} \cup \{ P_i^+(R)\ | \ R>0\} $ gives that every compact minimal surface $M$ with boundary
in $\mathcal{W}(\G ,r)$ lies in the domain $U_i(r)$; hence, $M\subset   U_1(r)\cap U_2(r)$.
If we prove that $U_1(r)\cap U_2(r)\subset \mathcal{W}(\G,2r)$, then we would deduce that
$M\subset \mathcal{W}(\G,2r)$, i.e. the proposition holds with $j=2$ in this case of $X$
admitting two distinct vertical planes which are subgroups of $X$. To check that
$U_1(r)\cap U_2(r)\subset \mathcal{W}(\G,2r)$, let $p=(x,y,z)\in X$;
we will denote by $p_1=(x,0,z)\in P_1$, $p_2=(0,y,z)\in P_2$ and $p_3=(0,0,z)\in \G $. Assume
that the following properties concerning the extrinsic distance $d$ in $X$ hold (we will prove them later):
\begin{enumerate}[(A)]
\item $d(p,p_i)=d(p,P_i)$, for $i=1,2$.
\item $d(p,p_3)=d(p,\G )$.
\end{enumerate}
Under these assumptions, we have:
\[
d(p,\G )=d(p,p_3)\leq d(p,p_1)+d(p_1,p_3)=d(p,P_1)+d(p_1,p_3)\stackrel{(\star)}{=}
d(p,P_1)+d(p,p_2)=d(p,P_1)+d(p,P_2),
\]
where in $(\star)$ we have left multiplied by $(0,y,0)$ (which is an ambient isometry of $X$, thus preserves distances)
in the second summand. Now the inclusion $U_1(r)\cap U_2(r)\subset \mathcal{W}(\G,2r)$ follows directly.
We next prove (A) and (B). Given $q\in X$ and $A\subset X$, let ${\mathcal C}(p,q)$ (resp. ${\mathcal C}(p,A)$)
be the set of piecewise smooth curves $\a \colon [0,1]\to X$ such that $\a (0)=p$ and $\a (1)=q$ (resp. $\a (1)\in A$).
Consider the maps
\[
\begin{array}{ll}
\Theta_1\colon {\mathcal C}(p,P_1)\to {\mathcal C}(p,p_1), & [\Theta_1(\a )](t)=(x,y(t),z),\\
\Theta_2 \colon {\mathcal C}(p,P_2)\to {\mathcal C}(p,p_2), & [\Theta_2 (\a )](t)=(x(t),y,z),\\
\Theta_3 \colon {\mathcal C}(p,\G )\to {\mathcal C}(p,p_3), & [\Theta_3 (\a )](t)=(x(t),y(t),z),
\end{array}
\]
if $\a (t)=(x(t),y(t),z(t))$, $t\in [0,1]$. Since $ e^{zA(b)}=\left(
\begin{array}{cc}
e^z & 0 \\
0 & e^{bz}
\end{array}\right) $, then equation~(\ref{eq:13}) implies that $\Theta _i$ decreases lengths of curves, for each
$i=1,2,3$. Thus, for $i=1,2$ we have
\[
d(p,P_i)= \inf_{\a \in {\mathcal C}(p,P_i)} \mbox{Length}(\a )\geq
\inf_{\a \in {\mathcal C}(p,P_i)} \mbox{Length}(\Theta _i(\a ))
\geq \inf_{\be \in {\mathcal C}(p,p_i)} \mbox{Length}(\be )=d(p,p_i).
\]
Since $p_i\in P_i$, then the last inequality is in fact an equality and (A) is proved. To prove (B) one
uses the same argument, changing $P_i$ by $\G $, $\Theta _i$ by $\Theta _3$ and $p_i$ by $p_3$.
This finishes the proof of the proposition when $X$ is not isomorphic to $\EE $ or  Nil$_3$.

Assume now that $X$ is isomorphic to $\EE $. Thus, after scaling the metric of $X$,
it is isomorphic and isometric to $\R^2\rtimes_{A(c)}\R $ with $A(c)=\left(
\begin{array}{cc}
0 & -c \\
1/c & 0
\end{array}\right) $ for some $c>0$ (see Section~2.7 of~\cite{mpe11}).
For $t\in \R$, define the vertical planes
$P_1(t)=\{x=t\}$, $P_2(t)=\{y=t\}$. For any $t>0$ and $i=1,2$,
let $U_i(t)$ be the slab in $X$ with boundary $P_i(-t)\cup P_i(t)$.
Equation~(\ref{eq:5}) gives that the left translation in $X$ by an element of the form
$(0,0,m \pi)$ with $m\in 2\Z $
writes as $ ({\bf 0},m\pi)*({\bf p}_2,z_2)=(e^{m\pi A(c)}\
{\bf p}_2,m\pi +z_2)=({\bf p}_2,m\pi +z_2)=({\bf p}_2,z_2)*({\bf 0},m\pi)$,
for all ${\bf p}_2\in \R^2$, $z_2\in \R$.
In particular, for all $t>0$ the slab $U_i(t)$ is invariant under the isometry which is given by
left (or right) translation by $(0,0,m \pi)$; note that $\cW (\G ,r)$ is also invariant under left
translation by $(0,0,m \pi)$, for all $r>0$.
Since for $r>0$ fixed, the family of sets
\[
\{ U_1(nr)\cap U_2(nr)\cap \{|z|\leq \pi\} \ | \ n\in \N\}
\]
 forms a compact exhaustion for the horizontal slab
 $\{|z|\leq \pi\}$  and $\cW (\G ,r)\cap \{|z|\leq \pi\}$ is compact, then there exists
a $k\in \N$ such that $\cW (\G ,r)\cap \{|z|\leq \pi\}\subset
U_1(kr)\cap U_2(kr)\cap \{|z|\leq \pi\}$.
Furthermore, this integer $k$ can be chosen independently from $r$ (because
the identity map from $\{|z|\leq \pi\}=\R^2\rtimes _A[-\pi ,\pi ]$
into $\R^2\times [-\pi,\pi]$ with the product metric is a quasi-isometry, for every
$A\in \mathcal{M}_2(\R )$).
The left-invariance property of $U_i(kr)$ and $\cW (\G ,r)$ by left translation by  $(0,0,2 \pi)$
implies that $\cW (\G ,r)$ is contained in $U_1(kr)\cap U_2(kr)$.
This implies that if $M\subset X$ is a compact minimal surface with $\partial M
\subset \cW (\G, r)$, then $\partial M\subset U_1(kr)\cap U_2(kr)$, and by the
maximum principle applied to the family of minimal surfaces $\{ P_i(t)\ | \ t\in \R \} $, $i=1,2$,
we deduce that $M\subset U_1(kr)\cap U_2(kr)$.  On the other hand, since the sets
$ \cW (\G ,r)\cap \{|z|\leq \pi\}$ also form
a compact exhaustion for the slab
$\{|z|\leq \pi\}$, then similar reasoning  shows that there exists a $j\in \N$ independent of $r$ such that
 $U_1(kr)\cap U_2(kr) \subset \cW (\G ,jr)$,  from where we conclude that
$M\subset \mathcal{W}(\G ,jr)$. This finishes the proof of the proposition in the case
$X$ is isomorphic to $\EE$.


 Suppose now that $X$ is isomorphic to Nil$_3$.  After a scaling of the metric,
 we may assume that  $X$ is $\R^2\rtimes_A\R $ with
 $A=\left(
\begin{array}{cc}
0 & 1 \\
0 & 0
\end{array}\right) $. 
In this case there exists a unique vertical plane which is a subgroup of $X$,
namely $P_1=\{ y=0\} $. Since Nil$_3$ is unimodular, Theorem~3.6 in~\cite{mpe11}
implies that the foliation of surfaces at constant
distance from $P_1$ consists of minimal surfaces in $X$.
Given $r>0$, let $P_1^{\pm }(r)=\{ y=\pm r\} $ be the boundary planes of
the closed regular neighborhood $U_1(r)$
of $P_1$ of radius $r$ (one can check that the distance from $\vec{0}$ to  $(0,\pm r,0)$ is $r$),
and let $S^{\pm }(r,R)\subset P_1^{\pm }(r)$ be the round circle of Euclidean radius $R>0$
centered at the point $(0,\pm r,0)$. {\bf We claim that for $R$ much larger than $r$,
there exists an embedded minimal annulus $A(r,R)\subset X$ with boundary
$\partial A(r,R)=S^+(r,R)\cup S^-(r,R)$}. To see this, first note that
\begin{enumerate}[(I)]
\item $U_1(r)$ is quasi-isometric to the Riemannian product
$\R^2\times [-r,r]$ under the mapping arising from normal coordinates on $P_1$.
This is because, in the fixed compact  regular neighborhood of radius 1 of the segment
$\{(0,t,0)\mid t\in [-r,r]\}$ in the slab  $U_1(r)$, the restricted mapping is a quasi-isometry with its
image and the differential of the normal coordinate map is invariant under left translations by elements in
$P_1$.
\item  The Euclidean area of the cylinder $C(r,R)=\{ (x,y,z)\ | \ (x,r,z)\in S^+(r,R), |y|\leq r\} $
is $4\pi rR$.
\end{enumerate}
From (I), (II) we deduce that the area in $X$ of $C(r,R)$ is less than $4\pi crR$, for some  $c>0$
that only depends on $r$.
As the union of the two disks $D^\pm(r,R)\subset P_1^{\pm }(r)$
bounded by $S^\pm (r,R)$ has area $2\pi R^2$ and each  of these disks is area-minimizing
in $X$ (in fact, $D^{\pm }(r,R)$ is the unique solution of the Plateau problem for
boundary $S^{\pm }(r,R)$ in $X$), then the {Douglas Criterion and the}
Geometric Dehn's Lemma for Planar Domains
in Theorem~5 in~\cite{my1} (as adapted in the more general boundary setting of~\cite{my2})  guarantee that for $R\gg r$, there exists
an embedded least-area annulus $A(r,R)$ in
$X$ with boundary $\partial A(r,R)=S^+(r,R)\cup S^-(r,R)$, and our claim is proved.

Consider the family of minimal annuli $\mathcal{F}=\{ p*A(r,R)\ | \ p\in P_1,\ [p*A(r,R)]\cap \cW(\G ,r)=\mbox{\O }\} $,
where $p*A(r,R)$ denotes the left translation of $A(r,R)$ by the element $p$. Observe that $\mathcal{F}$ satisfies the
following properties.
\begin{enumerate}[(${\mathcal F}$-I)]
\item ${\mathcal F}$ is nonempty. Furthermore, ${\mathcal F}$ is invariant under left translation by elements of $\G$
and under the rotation $R_{\G}$ by $\pi/2$ around $\G$ (this follows from the invariance of $\cW(\G ,r)$
and of $P_1$ under
these ambient isometries of $X$).
\item There exists $k>r$ such that for all $t\in [k,\infty )$, we have
\[
[(t,0,0)*(A(r,R))]\cap \cW (\G ,r)=\mbox{\O } \mbox{ and }[(-t,0,0)*R_{\G}(A(r,R))]\cap \cW (\G ,r)=\mbox{\O }.
\]
\end{enumerate}
Let  ${\mathcal F}'$ be the family of left translates of $A(r,R)$ and of $R_\G(A(r,R))$ appearing in
Property~(${\mathcal F}$-II) together with their left translates
by elements in $\G$. It follows that there exists a $j\in \N$
such that
\begin{equation}
\label{U1}
U_1(r)-\left( \cup _{F\in {\mathcal F}'}F\right) \subset \cW(\G,jr);
\end{equation}
the existence of $j$ follows from similar arguments as those
in the proof of Property~(I) above.

Finally, consider a compact minimal surface $M\subset X$ with
boundary $\partial M\subset \mathcal{W}(\G ,r)$. Since $\mathcal{W}(\G ,r)\subset U_1(r)$ and the boundary planes of
$U_1(r)$ are minimal, then the maximum principle implies that $M\subset U_1(r)$. To finish the proof of
{Proposition \ref{prop5.1}}, we will check that $M\subset \cW(\G,jr)$. Otherwise, as $M\subset U_1(r)$ then (\ref{U1}) implies
that $M$ intersects some annulus $F\in {\mathcal F}'$. But then by compactness of $M$ and $F$,
there exists a largest $t>0$
such that $M\cap [(t,0,0)*F] \neq \mbox{\O}$, which gives a contradiction to
the maximum principle since the boundary of $(t,0,0)*F$ is disjoint from $\partial M$.
Now the proof {of Proposition \ref{prop5.1}} is complete.
\end{proof}

\begin{theorem}
\label{ass:box5}
Let $X=\R^2\rtimes_A \R$ be a metric semidirect
product, and let $\G \subset X$ be a vertical geodesic.
Then, given $r,C>0$ there exists $R=R(r,C)>0$
such that for every compact immersed minimal surface
$M\subset \mathcal{W}(\G ,r)$ with the norm of its second fundamental form less than $C$,
the radius of $M$ is less than $R$.
\end{theorem}
\begin{proof}
After a fixed left translation of $\G$, we will assume that $\G$ is the $z$-axis in
 $\R^2\rtimes_A \R$.

To prove the theorem we proceed by contradiction.
Suppose that there exist $r,C>0$ and a sequence of compact,  immersed
minimal surfaces $h_n\colon M_n \la \mathcal{W}(\G ,r)$ with the norm of
their second fundamental forms less than $C$ and such that
there exist points $p_n\in M_n$ for which the intrinsic
distances from $p_n$ to the boundaries of
the $M_n$ satisfy
$d_{M_n}(p_n,\partial M_n)>n$ for all $n\in \N $.
Consider the compact domain $Y=\mathcal{W}(\G ,r)\cap[\R^2\rtimes_A\{0\}]$.
After left translating the immersions $h_n$ appropriately
by elements in the 1-parameter
subgroup $\G =\{ (0,0,s)\in \R^2\rtimes _A\R \ | \ s\in \R \} $
and passing to a subsequence,
we may assume that $h_n(p_n)\in Y$ for al $n$
and this sequence of points converges to
a point $q_{\infty}\in Y$.

Since the minimal immersions $h_n$ have uniform
curvature estimates, then there exists
a complete, connected immersed minimal surface
$h_{\infty}\colon M_{\infty} \la \mathcal{W}(\G ,r)$
of bounded second fundamental form that is a limit of the restriction of
(a subsequence, denoted in the same way, of)
the $h_n$ to certain smooth compact domains $\Omega _n\subset M_n$ with
$p_n\in \Omega _n$ and $d_n(p_n,\partial \Omega _n)>n$ for all $n\in \N$,
and such that $h_{\infty }(p_{\infty })=
q_{\infty }$ for some point $p_{\infty }\in M_{\infty }$.
Now consider the closure $\cM$ of the union of all left
translations of $h_{\infty }(M_{\infty })$ by elements in $\G $, i.e.,
\[
\cM=\ov{\{a* h_{\infty }(M_{\infty}) \mid a \in \G\}},
\]
which is a connected subset of $\mathcal{W}(\G ,r)$.
As $h_{\infty }(M_{\infty})$ has bounded
second fundamental form, then, by the same compactness arguments, given any point $q\in \cM $,
there exists a compact embedded minimal disk $D(q) \subset \cM$
with $q \in \Int(D(q))$.

We next consider the special case where $X$ is non-unimodular,
and so we assume $X=\R^2\rtimes_A \R$ with $A$ satisfying \eqref{Axieta};
see Remark~\ref{clasemi}.
Let $L$ be the line given by Corollary~\ref{corol5.4}.
For $p=(x,y,0)\in L$ with $x^2+y^2$ sufficiently large, the set $Y$
lies in the interior of the  strip $S\subset \R^2\rtimes_A \{0\}$ bounded by the
pair of Euclidean lines $L_p,L_{-p}\subset \R^2\rtimes_A \{0\}$
that are orthogonal to $L$ at the respective points $p,-p$.  By
Corollary~\ref{corol5.4}, there exist circles $C_p,C_{q=-p}$
with $p\in C_p$, $q\in C_q$ such that $C_p\cup C_q$ does not intersect the interior of the
strip $S$ and $C_p\cup C_q$ is the boundary of a
least-area embedded annulus $\Sigma\subset \R^2\rtimes_A [0,\infty)$.
Let  $L^\perp$ in  $\R^2\rtimes_A \{0\}$ be the line perpendicular to
$L$ at $\vec{0}$. For $t\in L^\perp$, let $t*\Sigma$ denote the left translation
of $\Sigma$ by $t$.  Note that for some $t_0\in \R$, $(t_0*\Sigma )\cap \cM\neq
\mbox{\rm \O}$
and that for $|t|$ large,  $(t*\Sigma )\cap \cM= \mbox{\rm \O}$.  It follows that there exists a
$t_1\in L^\perp$ with largest norm such that
$t_1*\Sigma $ intersects $\cM$ at some point $p_{t_1}$ and near $p_{t_1}$ the
set $\cM$ lies on one side of $t_1*\Sigma$.  This implies that there exists an embedded minimal disk
$D(p_{t_1})\subset \cM$ containing $p_{t_1}$, such that $D(p_{t_1})$
lies on one side of $t_1*\Sigma$. Now the maximum principle for minimal surfaces gives that
$t_1*\Sigma \subset \cM$, which is false since $\partial (t_1*\Sigma) \cap \cM=\mbox{\O}$.
This contradiction proves the proposition in this special case where $X$ is non-unimodular .

Finally, we consider the remaining  case where $X$ is unimodular (hence ${\rm tr} (A)=0$). Consider the  foliation of
$X$ by minimal planes produced in the proof of
Proposition~\ref{prop5.1}, i.e. (with the notation in that proposition),
the planes $P_1^-(R)$ at constant distance $R>0$
from $P_1$ in the case that $X$ is isomorphic to $\sol$, Nil$_3$ or $\rth$,
and the periodic planes $P_1(t)=\{ x=t\} $,
in the case of $X$ is isomorphic to $\EE$.  Since in all of these cases
there exist one of these minimal planes $P$ such that $P\cap \cM\neq\mbox{\O}$ and $\cM$ lies on one side of $P$,
the argument in the previous paragraph with $P$ in place of $\Sigma $
easily generalizes to give a contradiction. This contradiction
completes the proof of Theorem~\ref{ass:box5}.
\end{proof}

As a direct consequence of Theorem~\ref{ass:box5} and the classical curvature estimates
for stable minimal surfaces~\cite{ros9,sc3}, we obtain:

\begin{corollary}
\label{corrad}
Let $X=\R^2\rtimes_A \R$ be a metric semidirect
product,
and let $\cW(\G ,r)\subset X$ denote a solid metric cylinder in $X$ of radius $r>0$
around a vertical geodesic $\Gamma\subset X$. Then, there are no complete stable
minimal surfaces contained in $\cW(\G ,r)$.
\end{corollary}

In order to prove Theorem~\ref{th:intro} stated in the Introduction
we will need an auxiliary construction for the case that $X$ is
non-unimodular with positive Milnor $D$-invariant.
To do this, in the remainder of this section we fix
$\alfa\in[0,1)$ and $\beta\in[0,\infty)$, we consider the matrix $A=A(\alfa,\beta)$ given
by~(\ref{Axieta}), and the non-unimodular metric Lie group
$X=X(\alfa,\beta)=\R^2\rtimes_A\R$ with its usual left invariant metric
$\langle ,\rangle $ determined by $A$ (see Definition~\ref{def2.1}).
Under our hypotheses on
$\alfa,\beta$, we have that the Milnor $D$-invariant of $X$ is positive.
Conversely, every non-unimodular metric Lie
group with positive Milnor $D$-invariant can be expressed as $X(\alfa,\beta)$ for
some $\alfa\in [0,1)$ and $\beta \geq 0$, see Section~\ref{sec:background}.

The 1-parameter subgroup $\{(0,0,s)\in \R^2\rtimes_A \R\mid s\in
\R\}$ of $X$ generates under left multiplication a right invariant
vector field $F_3$ of $X$ (see equation~(\ref{eq:6}) where the
notation for the matrix $A$ is different from the one used here). Next we will
study the mean convexity of solid cylinders in $X$ obtained after flowing the domains enclosed
by a family of homothetic ellipses in $\R^2\rtimes _A\{ 0\} $  through the
1-parameter group
of isometries $\{ \phi _s\ | \ s\in \R \} $ associated to $F_3$,
namely the left translations by elements in the $z$-axis
of $\R^2\rtimes _A\R $.
The technical property stated in Proposition~\ref{Plateau2} will be used in Theorem~\ref{corrad2} below, in order
to obtain the desired radius estimate for stable minimal surfaces in $X$ that generalizes Corollary~\ref{corrad}.

Consider an ellipse $C_{\mu }=\{ (x,y,0)\in \R^2\rtimes _A\{ 0\} \ | \ x^2+ \frac{y^2}{\mu ^2}=1\} $,
where $\mu >0$ is to be determined, and the family of homothetic ellipses
\begin{equation}
\label{eq:ellipse} rC_{\mu }=\{(rx,ry,0)\mid (x,y,0)\in C_{\mu }\} , \qquad
r>0.
\end{equation}
Let $rE_\mu\subset \R^2\rtimes_A\{0\}$ denote the compact disk with boundary $rC_{\mu }$, and
let
\begin{equation}
\label{eq:F3cyl}
\Omega (r)=\bigcup_{s\in \R}\phi _s(rE_{\mu })
\end{equation}
be the $F_3$-invariant closed solid cylinder obtained after flowing $rE_{\mu }$
by the isometries that generate $F_3$.
\begin{proposition}
\label{Plateau2}
Let $X=\R^2\rtimes_A \R$ be a metric semidirect
product where $A$ is as in equation (\ref{Axieta}) with $\a\in[0,1)$
and $\be \in[0,\infty)$. Then, there exist $\mu >0$ and $r_0>0$ such that the
$F_3$-invariant solid cylinder $\Omega (r)$ over $rE_{\mu }$ defined in~(\ref{eq:F3cyl})
is strictly mean convex for every $r\geq r_0$.
\end{proposition}
\begin{proof}
Fix a positive $\mu $ to be determined later. Given $r>0$,
parameterize $rC_{\mu }$ by $\gamma=\gamma(t)=(x(t),y(t),0)$, where
\begin{equation}
\label{eq:g}
x(t)=r\cos t , \quad
y(t)=r\mu \sin t,\qquad t\in [0,2\pi ].
\end{equation}
A parametrization $\Phi $ of the $F_3$-invariant
cylinder given by the boundary $\Sigma =\Sigma (r)$ of $\Omega (r)$
is obtained by flowing $\g $ through the 1-parameter group $\{ \phi _s
\ | \ s\in \R \} $, i.e.,
\[
\Phi (t,s)=\phi _s(\g (t)), \quad (t,s)\in [0,2\pi ]\times \R ,
\]
 where $\phi_s({\bf p},z)=(e^{sA}{\bf p},s+z)$ for all $s,z\in \R$ and ${\bf
p}\in \R^2$ (${\bf p}$ is considered to be a column vector).

The mean curvature $H=H(t,s)$ of $\Sigma $ is given by the
well-known formula
\begin{equation}
\label{eq:Hcyl} 2(EG-F^2)H=eG-2fF+gE,
\end{equation}
where $E,F,G$ and $e,f,g$ are respectively the coefficients of the
first and second fundamental form of $\Sigma $ (these coefficients are
functions of $(t,s)$):
\begin{equation}
\label{eq:I,II}
\begin{array}{ccc}
E=\| \Phi _t\| ^2,& F=\langle \Phi _t,\Phi _s\rangle ,& G=\| \Phi _s\| ^2,
\\
e=\langle N,\nabla _{\Phi _t}\Phi _t\rangle ,& f=\langle N,\nabla_{\Phi _t}\Phi _s\rangle ,
& g=\langle N,\nabla _{\Phi _s}\Phi _s\rangle ,
\end{array}
\end{equation}
where $\Phi _t=\frac{\partial \Phi }{\partial t}$, $\Phi _s=\frac{\partial \Phi }{\partial s}$ and
$N=\frac{\Phi _t\times \Phi _s}{\| \Phi _t\times \Phi _s\| }$ is the unit normal vector field to $\Sigma $.
Observe that $\Phi _t(t,0)$ defines the counterclockwise orientation on $rC_{\mu }$
 and that $\Phi _s(t,0)$ points upward. Therefore, $N(t,0)$ points outward $\Omega (r)$
  along~$\g $.
Since $\Sigma $ is $F_3$-invariant, the strict mean convexity of $\Omega (r)$ will follow from
the existence of some $\mu >0$ (depending solely on $\a,\be $) such that the function
$t\in [0,2\pi ]\mapsto (eG-2fF+gE)(t,0)$ is strictly negative for $r>0$ large enough.

Note that
\begin{equation}
\label{Phit}
\Phi _t(t,0)=\g '(t)=\left(
\begin{array}{c}
x'(t)
\\
y'(t)
\\
0
\end{array}
\right) =
\left[
\begin{array}{c}
x'(t)
\\
y'(t)
\\
0
\end{array}
\right]
=\left[
\begin{array}{c}
-\frac{1}{\mu }y(t)
\\
\mu x(t)
\\
0
\end{array}
\right] ,
\end{equation}
where the parentheses (resp. brackets) refer to coordinates with respect
to the basis $\{ \partial _x,\partial _y,\partial _z\} $ (resp. to
the usual orthonormal basis $\{ E_1,E_2,E_3\} $ of the Lie algebra
of $X$ given by~(\ref{eq:6*})); in general, the change of coordinates between the two
bases at a point $(x,y,z)\in \R^2\rtimes _A\R $ is
\begin{equation}
\label{changebasis}
\left[
\begin{array}{c}
a
\\
b
\\
c
\end{array}
\right] =
\left(
\begin{array}{c}
e^{zA}\left( \begin{array}{c}
a
\\
b
\end{array}\right)
\\
c
\end{array}
\right) , \quad a,b,c\in \R.
\end{equation}
Also,
$\Phi _s(t,s)=(F_3)_{\Phi (t,s)}$ and so, the globally defined right
invariant vector field $F_3$ extends $\Phi _s$. Using (\ref{eq:6}) (recall
that the entries of the matrix $A$ are given by (\ref{Axieta})), we
have
\begin{equation}
\label{eq:F3bis}
F_3(x,y,z)=\left(
\begin{array}{c}
\de
\\
\ve
\\
1
\end{array}
\right)
\stackrel{(\ref{changebasis})}{=}
\left[
\begin{array}{c}
\de a_{11}(-z)+\ve a_{12}(-z)
\\
\de a_{21}(-z)+\ve a_{22}(-z)
\\
1
\end{array}
\right],
\end{equation}
where
\begin{equation}
\label{eq:deve}
\de (x,y)=(1+\a )x-(1-\a )\be y,\quad
\ve (x,y)=(1+\a )\be x+(1-\a )y,
\end{equation}
and $a_{ij}(z)$ are the entries of the matrix $e^{zA}$, see
(\ref{eq:exp(zA)}). In particular,
\begin{equation}
\label{Phis}
\Phi _s(t,0)=(F_3)_{\g (t)}=
\left[
\begin{array}{c}
\de (t)
\\
\ve (t)
\\
1
\end{array}
\right] ,
\end{equation}
where $\de (t)=\de (\g (t))$ and $\ve (t)=\ve (\g (t))$.

From (\ref{Phit}), (\ref{Phis}) we can compute the coefficients of the first fundamental form at
points of the form $\Phi (t,0)$:
\begin{equation}
\label{1stff}
\left\{
\begin{array}{rcl}
E(t,0)&=&x'(t)^2+y'(t)^2,
\\
F(t,0)&=&\de (t)x'(t)+\ve (t)y'(t),
\\
G(t,0)&=&1+\de (t)^2+\ve (t)^2.
\end{array}
\right.
\end{equation}
The unit normal vector field at points of the form $\Phi (t,0)$ is given by
\begin{equation}
\label{Ncyl} N(t,0)=\frac{1}{\Delta (t)}(\Phi _t\times \Phi _s)(t,0)=
\frac{1}{\Delta (t)} \left[
\begin{array}{c}
y'(t)
\\
-x'(t)
\\
\ve (t)x'(t)-\de (t)y'(t)
\end{array}
\right],
\end{equation}
where $\Delta (t)=\| \Phi _t\times \Phi _s\| (t,0)$.

We next compute the coefficients of the second fundamental form of $\Sigma$. Using (\ref{Phit})
and denoting by $\frac{DW}{dt}$ the covariant
derivative of a vector field $W$ along $\g $, we have
\[
\label{eq:nablaXtXt}
 \left( \nabla _{\Phi _t}\Phi _t\right) (t,0)\stackrel{(\ref{Phit})}{=}
 \frac{D}{dt}\left( x'(t)(E_1)_{\g (t)}+y'(t)(E_2)_{\g (t)}\right)
\]
\[
=\left[
\begin{array}{c}
x''(t)
\\
y''(t)
\\
0
\end{array}
\right]
+x'(t)\nabla _{\g'(t)}E_1+y'(t)\nabla _{\g'(t)}E_2.
\]
\begin{equation}
\label{Phitt}
\stackrel{(\ref{eq:12})}{=}
\left[
\begin{array}{c}
x''(t)
\\
y''(t)
\\
(1+\a )x'(t)^2+2\a \be x'(t)y'(t)+(1-\a )y'(t)^2
\end{array}
\right].
\end{equation}

Analogously,
\[
 \left( \nabla _{\Phi _t}\Phi _s\right) (t,0)= \frac{D(F_3\circ \g )}{dt}
 \stackrel{(\ref{Phis})}{=}
 \frac{D}{dt}\left( \de (t)(E_1)_{\g (t)}+\ve (t)(E_2)_{\g (t)}
 +(E_3)_{\g (t)}\right)
 \]
 \[
 =\left[
\begin{array}{c}
\de '(t)
\\
\ve '(t)
\\
0
\end{array}
\right]
+\de (t)\nabla _{\g'(t)}E_1+\ve (t)\nabla _{\g'(t)}E_2+
\nabla _{\g'(t)}E_3
\]
\begin{equation}
\label{eq:nablaXtXs}
\stackrel{(\ref{eq:12})}{=}
\left[
\begin{array}{c}
\de '(t)-(1+\a )x'(t)-\a \be y'(t)
\\
\ve '(t)-\a \be x'(t)-(1-\a )y'(t)\\
\de (t)[(1+\a )x'(t)+\a \be y'(t)]+\ve (t)[\a \be x'(t)+(1-\a )y'(t)]
\end{array}
\right] .
\end{equation}

To compute $g=\langle N,\nabla _{\Phi _s}\Phi _s\rangle = \langle N,\nabla
_{\Phi _s}F_3\rangle $ we use that $F_3$ is a Killing vector field that extends $\Phi_s$:
\[
g(t,s)=-\langle \Phi _s,\nabla _{N}F_3\rangle =-\frac{1}{2}N\left( \|
F_3\|^2\right)
\]
\[
\stackrel{(\ref{eq:F3bis})}{=} -\left( \de a_{11}(-z)+\ve a_{12}(-z)\right)
N\left( \de a_{11}(-z)+\ve a_{12}(-z)\right)
\]
\[
-\left( \de
a_{21}(-z)+\ve a_{22}(-z)\right)
N\left( \de a_{21}(-z)+\ve a_{22}(-z)\right) .
\]
Hence,
\begin{equation}
\label{g(t,0)}
g(t,0)= -\de (t)\left\{  N(\de )+\de (t)N\left( a_{11}(-z)\right) +\ve
(t)N\left( a_{12}(-z)\right) \right\}
\end{equation}
\[
-\ve (t)\left\{  N(\ve )+\de (t)N\left( a_{21}(-z)\right) +\ve (t)N \left(
a_{22}(-z)\right) \right\},
\]
where we have simplified the notation $N(t,0)$ by $N$.
By using (\ref{eq:deve}) and (\ref{Ncyl}) one has (at $(t,0)$):
\begin{equation}
\label{Nderv}
\left\{
\begin{array}{rcl}
N(\de )&=&\frac{1}{\Delta (t)}\left[ (1+\a )y'(t)+(1-\a )\be x'(t)\right] ,
\\
N(\ve )&=&{\textstyle \frac{1}{\Delta (t)}}\left[
(1+\a )\be y'(t)-(1-\a )x'(t)\right] ,
\\
N(a_{ij}(-z))&=&-{\textstyle \frac{1}{\Delta (t)}}\left[ \ve (t)x'(t)-\de
(t)y'(t)\right] a_{ij}'(0),\quad \mbox{for }\ i=1,2.
\end{array}
\right.
\end{equation}
Since $(a_{ij}(z))_{i,j}=e^{zA}$, then
$(a'_{ij}(0))_{i,j}=A$. Now, (\ref{g(t,0)}) and (\ref{Nderv})  give
\begin{equation}
\label{eq:Deltag}
\begin{array}{rcl}
\Delta (t)g(t,0)\hspace{-.2cm}&=&\hspace{-.2cm} -\left\{ (1+\a )y'+(1-\a )\be x'
-\left[ (1+\a )\de -(1-\a )\be \ve \right] (\ve x'-\de y')\right\} \de
\\
& &\hspace{-.2cm} -\left\{ (1+\a )\be y'-(1-\a )x'-\left[ (1+\a )\be \de +(1-\a )\ve \right]
(\ve x'-\de y')\right\} \ve .
\end{array}
\end{equation}

A direct substitution from (\ref{eq:g}), (\ref{eq:I,II}), (\ref{eq:deve}), (\ref{Phitt}),
(\ref{eq:nablaXtXs}) and (\ref{eq:Deltag}) gives that
\[
\begin{array}{rcl}

\Delta (t)\, (eG-2fF+gE)(t,0)\hspace{-.2cm} &=&\hspace{-.2cm}-\mu r^2
+r^4h_{\a ,\be ,\mu }(t)
\\
\rule{0cm}{.5cm} & &\hspace{-.2cm} -\frac{r^6}{4}\left\{
2\mu +2\a \mu \cos(2t)+\be [1+\a -(1-\a )\mu ^2]\sin(2t)\right\}^3,
\end{array}
\]
where $h_{\a ,\be ,\mu }(t)$ is a smooth, $\pi $-periodic function of $t$
depending on the parameters $\a ,\be ,\mu $.
From the last displayed expression we deduce that the mean curvature of $\Sigma $
with respect to $N$ is strictly negative for all $r$ large enough provided that the expression
\[
\varrho _{\a ,\be ,\mu }(t)=2\mu +2\a \mu \cos(2t)+\be [1+\a -(1-\a )\mu ^2]\sin(2t)
\]
is positive as a function of $t\in [0,2\pi ]$, for any given
values $\a \in [0,1)$, $\be \in [0,\infty )$
and for some choice of $\mu =\mu (\a ,\be )>0$. Clearly,
\[
\varrho _{\a ,\be ,\mu }(t)=2\mu +\langle u,v(t)\rangle \geq 2\mu -\|u\| ,
\]
where $u=u(\a ,\be ,\mu )=\left( 2\a \mu ,\be [1+\a -(1-\a )\mu ^2]\right)$,
$v(t)=\left( \cos(2t),\sin(2t)\right)\in \R^2$ and both the last inner product and norm
refer to the usual flat metric in $\R^2$.
 Therefore, the proposition will be proved if we show that the following elementary
property holds:
\begin{enumerate}[(R)]
\item Given $(\a ,\be )\in [0,1)\times [0,\infty )$, there exists $\mu >0$ such that
 \[
 4\mu ^2>\| u\| ^2=4\a ^2\mu ^2+\be ^2[1+\a -(1-\a )\mu ^2]^2.
 \]
\end{enumerate}
If $\be =0$, then Property (R) clearly holds as $\a ^2<1$. If $\be >0$, then
the proof of Property (R) follows from an elementary analysis of the
function $\chi (\l )=4(1-\a ^2)\l -\be ^2[1+\a -(1-\a )\l ]^2$, which has a (unique)
maximum at $\l _0=\frac{1+\a }{1-\a }\left( 1+2\be ^{-2}\right) >0$, with
value $\chi (\l _0)=4(1+\a )^2\left( 1+\be ^{-2}\right) >0$. Now the desired
$\mu >0$ can be chosen as $\mu =\sqrt{\l _0}$. This completes the proof
of the proposition.
\end{proof}


As a consequence of the mean convexity of $\Omega(r)$, we have:

\begin{proposition}
\label{ass:box2}
Let $X=\R^2\rtimes_A \R$ be a non-unimodular metric semidirect
product, where $A$ is as in equation (\ref{Axieta}) with $\a \in[0,1)$
and $\be \in[0,\infty)$. Let $\mu ,r_0>0$ and $\Omega (r)$ be the numbers
and related $F_3$-invariant, mean convex solid cylinder given in
Proposition~\ref{Plateau2}. Suppose that
$r>r_0$.
Then every compact immersed minimal surface $M\subset X$ whose boundary lies in
   $\Omega (r)$ satisfies that $M\subset \Omega (r)$.
\end{proposition}
\begin{proof}
It is a consequence of a standard mean curvature comparison argument
based on the following two facts:
 \begin{enumerate}[$\bullet $]
 \item For every $r\geq r_0$, $\Omega(r)$ has
 mean convex boundary by Proposition~\ref{Plateau2}, and
 $\Omega(r_0)\subset\Omega(r)$.
\item The collection of boundaries $\{ \partial \Omega (r)\ | \ r\geq r_0\} $ forms a codimension-one foliation
of $X-\Omega (r_0)$.
\end{enumerate}
\end{proof}
Finally, from Proposition~\ref{prop5.1}, Theorem~\ref{ass:box5} and Proposition~\ref{ass:box2}, we can conclude the desired
radius estimate for compact stable minimal surfaces in metric semidirect products stated in Theorem~\ref{th:intro}:

\begin{theorem}\label{corrad2}
Let $X=\R^2\rtimes_A \R$ be a metric semidirect
product. Given $r>0$ and any vertical geodesic $\G\subset X$, there exists a
positive number $\Lambda(r)>0$ such that the following property holds: for any compact stable
minimal surface $M$ in $X$ such that all points of its boundary $\parc M$ are at distance at
most $r$ from $\Gamma$, the radius of $M$ is at most $\Lambda(r)$.
\end{theorem}
\begin{proof}
We first consider the case  where $A$ is as in equation (\ref{Axieta}) with $\alfa\in[0,1)$
and $\beta\in[0,\infty)$.
Let $M$ be a compact stable minimal surface in $X$ whose boundary lies in the closed
solid metric cylinder $\cW(\G ,r)$ of radius $r>0$ in $X$ around a vertical geodesic $\Gamma$.
Note that there
exists some $r'=r'(r)>0$ such that $\cW(\G ,r)\subset \Omega(r')$, where $\Omega(r')$ is the mean
convex solid cylinder given in Proposition~\ref{Plateau2}. Then, by Proposition~\ref{ass:box2}
we deduce that $M\subset \Omega(r')$. As there exists some $r''=r''(r)>0$ such
that $\Omega(r')\subset \cW(\Gamma,r'')$,
we obtain from Theorem~\ref{ass:box5} and the Schoen-Ros curvature estimates for stable minimal
surfaces~\cite{sc3,ros9} the desired radius estimate.

If $X$ is non-unimodular with $A$ given by \eqref{Axieta} and
non-positive Milnor $D$-invariant, or else trace$(A)=0$, then we
apply Proposition~\ref{prop5.1} to conclude that every compact immersed stable
minimal surface in $X$ is contained in some metric cylinder
$\cW(\G,r')$, where $r'$ depends only on $X$ and $r$. Then, as in the previous paragraph,
Theorem~\ref{ass:box5} implies that $M$ has a radius estimate that only depends on $X$ and $r$.
This last observation completes the proof.
\end{proof}

\section{Appendix.}

\begin{proposition}
\label{propos2app}
Let $X$ be a non-unimodular semidirect product
$\R^2\rtimes _A\R $ endowed
with its canonical metric, where
$A\in \mathcal{M}_2(\R )$ is given by~(\ref{Axieta}) for
some constants $\alfa,\beta\geq 0$. Let $D=\det(A)=(1-\alfa^2)(1+\beta^2)$ be the Milnor
$D$-invariant associated to the Lie group $\R^2\rtimes _A\R $.
Then, the following properties hold:
\ben[1.]
\item Given $z\in \R $, the exponential of the matrix $zA$ is equal to
\begin{equation}
\label{eq:exp(zA)2}
e^{zA}=e^z\left[ {\bf C}_D(z)\, I_2+{\bf S}_D(z)(A-I_2)\right] ,
\end{equation}
where $I_2\in \mathcal{M}_2(\R )$ is the identity matrix and
\begin{equation}
\label{eq:S,C}
\mathbf{C}_D(t)=\left\{ \begin{array}{cl}
\cosh (\sqrt{1-D}\, t) & \mbox{ if } D<1,\\
1 & \mbox{ if } D=1,\\
\cos (\sqrt{D-1}\, t) & \mbox{ if } D>1,
\end{array}\right.
\
\mathbf{S}_D(t)=\left\{ \begin{array}{cl}
\frac{1}{\sqrt{1-D}}\sinh (\sqrt{1-D}\, t) & \mbox{ if } D<1,\\
t & \mbox{ if } D=1,\\
\frac{1}{\sqrt{D-1}}\sin (\sqrt{D-1}\, t) & \mbox{ if } D>1.
\end{array}\right.
\end{equation}
\item The norms of $\partial _x$, $\partial _y$
and their inner product with respect to
the canonical metric are
\[
\begin{array}{rcl}
\| \partial _x\| ^2&=&e^{-2z}\left\{ \beta^2(1+\alfa)^2{\bf S}_D(z)^2+\left[
{\bf C}_D(z)-\alfa\, {\bf S}_D(z)\right] ^2 \right\}
\\
\rule{0cm}{.5cm}
 \| \partial _y\| ^2&=&e^{-2z}\left\{ \beta^2(1-\alfa)^2{\bf
S}_D(z)^2+\left[ {\bf C}_D(z)+\alfa\, {\bf S}_D(z)\right] ^2 \right\}
\\
\rule{0cm}{.5cm}
\langle \partial _x,\partial _y\rangle &=&-2\alfa\beta e^{-2z}\,
{\bf S}_D(z)\left[ {\bf S}_D(z)+{\bf C}_D(z)\right] .
\end{array}
\]
In particular:
\ben[2a.]
\item if $D>0$ then $\| \partial _x\| ^2,\| \partial _y\|
^2,\langle \partial _x,\partial _y\rangle$ decay exponentially as
$z\to +\infty $, and so, the norm of every horizontal right
invariant  vector field in $X$ decays exponentially as $z\to +\infty
$ as well.
\item If $D<1$, then $A$ is diagonalizable with distinct eigenvalues $\l
_{\pm }=1\pm \sqrt{1-D}$. Let $v_+,v_-\in \R^2$ be unitary
eigenvectors of $A$ associated to $\l _+,\l _-$, and let $V_+,V_-$
be the horizontal right invariant vector fields in $X$ determined by
$V_{\pm }(\vec{0})=v_{\pm }$. Then, $\| V_{\pm }\|
(x,y,z)=e^{-\l _{\pm }z}$ for all $(x,y,z)\in X$.
 \end{enumerate}
\end{enumerate}
\end{proposition}
\begin{proof}
Since $\mbox{\rm trace}(A)=2$ and $\det (A)=D$, then the characteristic equation for $A$ gives
$A^2-2\, A+D\, I_2=0$.
From here it is straightforward to show that if
we define $f\colon \R \to \mathcal{M}_2(\R )$
by $f(z)=e^z\left[ {\bf C}_D(z)\, I_2+{\bf S}_D(z)(A-I_2)\right] $,
then $f'(z)=Af(z)$ and $f(0)=I_2$
(for this, use that ${\bf S}_D'={\bf C}_D$ and that ${\bf C}_D'=(1-D){\bf S}_D$),
which gives item~1 of the proposition.

The three displayed equalities in item~2 of the proposition are also direct
computations that only use~(\ref{eq:exp(zA)2}) and the
expression~(\ref{eq:13}) of the canonical metric in terms of
$x,y,z$. If $D>0$, then (\ref{eq:S,C}) and the three
displayed equalities in item~2 imply that $\| \partial _x\| ^2,\|
\partial _y\| ^2,\langle \partial _x,\partial _y\rangle$ decay
exponentially as $z\to +\infty $. If $D<1$, then the characteristic
equation of $A$ has two distinct real roots $\l _{\pm }=1\pm
\sqrt{1-D}$, which implies that $A$ is diagonalizable. After a fixed
rotation in the $(x,y)$-plane around the origin (this change of
coordinates does not affect either the Lie group structure in
$\R^2\rtimes _A\R $ or  its canonical metric), we can assume that
$V_+=\partial _x$, i.e.,
\[
A=\left(
\begin{array}{cr}
\l _+ & b \\
0 & \l _- \end{array}\right) \mbox{ and thus, }
e^{zA}=\left(
\begin{array}{cr}
e^{\l _+ z} & a_{12}(z) \\
0 & e^{\l _- z}\end{array}\right)
\]
for certain $b\in \R $ and $a_{12}(z)$ function of $z$.
Thus, (\ref{eq:13}) directly gives that $\| V_+\| (x,y,z)$
$=\| \partial _x\| (x,y,z)=e^{-\l _+z}$. The proof of
 $\| V_-\| (x,y,z)=e^{-\l _-z}$ is analogous.
\end{proof}

\vspace{.3cm} \noindent {\sc Acknowledgments:}
First author's financial support: This material is based upon work for
the NSF under Award No. DMS-1309236. Any opinions, findings, and
conclusions or recommendations expressed in this publication are those
of the authors and do not necessarily reflect the views of the NSF.
Second author's financial support:  Research partially supported by MICINN-FEDER,
Grant No.  MTM2013-43970-P, and Programa de Apoyo a la Investigacion,
Fundacion Seneca-Agencia de Ciencia y Tecnologia
Region de Murcia, reference 19461/PI/14.
Third  author's financial support: Research partially supported by a
{MINECO/FEDER
grant no. MTM2014-52368-P}.

\bibliographystyle{plain}
\bibliography{bill}
\end{document}